\documentclass[11pt]{article} % use larger type; default would be 10pt
\usepackage[utf8]{inputenc} % set input encoding (not needed with XeLaTeX)

\usepackage[margin=1in]{geometry} % to change the page dimensions
\usepackage{graphicx} % support the \includegraphics command and options
\usepackage{booktabs} % for much better looking tables	
\usepackage{array} % for better arrays (eg matrices) in maths
\usepackage{paralist} % very flexible & customisable lists (eg. enumerate/itemize, etc.)
\usepackage{verbatim} % adds environment for commenting out blocks of text & for better verbatim
\usepackage{mathrsfs}
\usepackage{amssymb}
\usepackage{amsthm}
\usepackage{amsmath,amsfonts,amssymb}
\usepackage{esint}
\usepackage{graphics}
\usepackage{enumerate}
\usepackage{mathtools}
\usepackage{xfrac}
\usepackage{subcaption}
\usepackage{stmaryrd}
\SetSymbolFont{stmry}{bold}{U}{stmry}{m}{n}
\usepackage{mathabx}

\usepackage[dvipsnames]{xcolor}
\usepackage[colorlinks=true, pdfstartview=FitV, linkcolor=blue, citecolor=blue, urlcolor=blue]{hyperref}
\usepackage[normalem]{ulem}

\usepackage{tikz}
\usetikzlibrary{calc}
\usepackage{pgf}
\usetikzlibrary{external}
\numberwithin{equation}{section}
\numberwithin{figure}{section}

\newtheorem{theorem}{Theorem}[section]

\newtheorem{assumption}[theorem]{Assumption}
\newtheorem{corollary}[theorem]{Corollary}
\newtheorem{proposition}[theorem]{Proposition}
\newtheorem{lemma}[theorem]{Lemma}
\theoremstyle{definition}

\newtheorem{remark}[theorem]{Remark}

\newcommand*{\Id}{\ensuremath{\mathrm{I}_d}}

\newcommand*{\N}{\ensuremath{\mathbb{N}}}

\newcommand*{\Z}{\ensuremath{\mathbb{Z}}}

\newcommand*{\R}{\ensuremath{\mathbb{R}}}

\newcommand{\eps}{\varepsilon}

\renewcommand{\P}{\ensuremath{\mathbb{P}}}

\newcommand{\g}{\mathbf{g}}
\renewcommand{\c}{\mathsf{c}}

\newcommand{\s}{\mathbf{s}}
\renewcommand{\t}{\mathbf{t}}

\DeclareSymbolFont{boldoperators}{OT1}{cmr}{bx}{n}
\SetSymbolFont{boldoperators}{bold}{OT1}{cmr}{bx}{n}
\usepackage{accents}

\newcommand{\K}{\mathbf{k}}

\newcommand{\T}{\mathbb{T}}
\newcommand{\Cs}{\mathscr{C}}

\def\XXint#1#2#3{{\setbox0=\hbox{$#1{#2#3}{\int}$}
\vcenter{\hbox{$#2#3$}}\kern-.5\wd0}}

\let\originalleft\left
\let\originalright\right
\renewcommand{\left}{\mathopen{}\mathclose\bgroup\originalleft}
\renewcommand{\right}{\aftergroup\egroup\originalright}

\newcommand{\indc}{{\boldsymbol{1}}}

\newcommand{\C}{\mathscr{C}}
\newcommand{\A}{\mathscr{A}}
\newcommand{\E}{\mathbb{E}}
\newcommand{\Pc}{\mathcal{P}}
\newcommand{\Mc}{\mathcal{M}}
\newcommand{\ux}{\underline{x}}
\newcommand{\rr}{\mathsf{r}}

\newcommand{\Ec}{\mathcal{E}}
\newcommand{\Dc}{\mathcal{D}}
\newcommand{\diff}{\mathrm{d}}

\renewcommand{\H}{\dot H^{\frac{\s-d}{2}}}
\newcommand{\Hp}{\dot H^{1+\frac{\s-d}{2}}}
\renewcommand{\hat}{\widehat}

\makeatletter
\pgfmathdeclarefunction{erf}{1}{%
  \begingroup
    \pgfmathparse{#1 > 0 ? 1 : -1}%
    \edef\sign{\pgfmathresult}%
    \pgfmathparse{abs(#1)}%
    \edef\x{\pgfmathresult}%
    \pgfmathparse{1/(1+0.3275911*\x)}%
    \edef\t{\pgfmathresult}%
    \pgfmathparse{%
      1 - (((((1.061405429*\t -1.453152027)*\t) + 1.421413741)*\t 
      -0.284496736)*\t + 0.254829592)*\t*exp(-(\x*\x))}%
    \edef\y{\pgfmathresult}%
    \pgfmathparse{(\sign)*\y}%
    \pgfmath@smuggleone\pgfmathresult%
  \endgroup
}
\makeatother

\usepackage{titlesec}

\newcommand{\addperiod}[1]{#1.}
\titleformat{\section}
   {\centering\normalfont\Large}{\thesection.}{0.5em}{}
\titleformat*{\subsection}{\bfseries}
\titleformat{\subsubsection}[runin]
  {\normalfont\bfseries}
  {\thesubsubsection.}
  {0.5em}
  {\addperiod}
\titleformat*{\subsubsection}{\normalfont\itshape}
\titleformat*{\paragraph}{\bfseries}
\titleformat*{\subparagraph}{\large\bfseries}

\title{Large deviation principles for singular Riesz-type diffusive flows}

\author{Elias Hess-Childs
\thanks{Courant Institute of Mathematical Sciences, New York University.
{\footnotesize \href{elias.hess-childs@courant.nyu.edu}{elias.hess-childs@courant.nyu.edu}.}
}
}
\date{ } 

\usepackage[nottoc,notlot,notlof]{tocbibind}

\begin{document}

\maketitle

\begin{abstract}
We combine hydrodynamic and modulated energy techniques to study the large deviations of systems of particles with pairwise singular repulsive interactions and additive noise. Specifically, we examine periodic Riesz interactions indexed by parameter $\mathbf{s}\in[0,d-2)$ for $d\geq 3$ on the torus. When $\mathbf{s}\in(0,d-2)$, we establish a large deviation principle (LDP) upper bound and partial lower bound given sufficiently strong convergence of the initial conditions. When $\mathbf{s}=0$ (i.e., the interaction potential is logarithmic), we prove that a complete LDP holds. Additionally, we show a local LDP holds in the distance defined by the modulated energy.
\end{abstract}

\section{Introduction}
 We study the large deviations as $N\rightarrow\infty$ of the systems of interacting particles given by
\begin{equation}\label{eq:SDE}
\begin{cases}
\diff x_i^t=-\frac{1}{N}\sum\limits_{1\leq j\leq N:j\neq i}\nabla\g(x_i^t-x_j^t)\,\diff t +\sqrt{2\sigma}\,\diff w^t_i,\\
x_i^t|_{t=0}=x_i^0,
\end{cases}\qquad i\in\{1,\dotsc,N\}.
\end{equation}
Above, $w_i$ are $N$ independent standard Brownian motions in the $d$-dimensional torus $\T^d$, the initial conditions $x_i^0$ are deterministic, $\sigma>0$ is the temperature of the system, and $\g$ is a sub-Coulomb periodic Riesz interaction. That is to say, $\g$ is the unique zero-mean periodic solution to
\[(-\Delta)^{\frac{d-\s}{2}}\g=\c_{d,\s}(\delta_0-1),\]
where $\s\in[0,d-2)$ and the choice of the scaling constant is made so $\g$ behaves like $|x|^{-\s}$ or $-\log|x|$ near the origin when $\s>0$ or $\s=0$ respectively. Letting $\Pc(\T^d)$ be the space of probability measures topologized by weak convergence, we show that for any fixed time horizon $T>0$ the \textit{empirical trajectories} 
\[\mu_N:=t\mapsto\frac{1}{N}\sum_{i=1}^N\delta_{x_i^t},\]
viewed as random elements of $C([0,T],\Pc(\T^d))$, satisfy large deviation estimates when the initial conditions strongly converge to some $\gamma\in \Pc(\T^d)\cap L^\infty(\T^d)$, see Assumption~\ref{cond:initial-convergence}. Specifically, we show that $I_\gamma:C([0,T],\Pc(\T^d))\rightarrow[0,\infty]$ defined in~\eqref{def:rate_function} has compact sublevel sets and that
\begin{align*}
-\inf_{\mu\in B^{\circ}\cap \A}I_\gamma(\mu)&\leq\liminf_{N\rightarrow\infty}\frac{1}{N}\log\P(\mu_N\in B)\leq\limsup_{N\rightarrow\infty}\frac{1}{N}\log\P(\mu_N\in B)\leq -\inf_{\mu\in\overline{B}}I_\gamma(\mu),
\end{align*}
for all Borel $B$, where $\A\subset C([0,T],\Pc(\T^d))$ is the dense subset defined in~\eqref{def:dense_set}. When $\s=0$, $\mu\in\A$ whenever $I_\gamma(\mu)<\infty$, thus $\mu_N$ satisfy an LDP with good rate function $I_\gamma.$ When $\mu\in L^\infty([0,T],L^\infty(\T^d))$ we also show the local estimate
\[
\P\bigg(\sup_{t\in[0,T]} F_N(\ux_N^t,\mu^t)<\eps\bigg)\approx \exp\Big(-N \big(I_\gamma(\mu)+o_\eps(1)\big)\Big),\]
where $F_N(\ux^t_N,\mu^t)$ is the modulated energy defined in~\eqref{eq:modulated-energy}.

\subsection{Background}

The system~\eqref{eq:SDE} corresponds to dissipative dynamics with respect to the energy
\[ H_N(\ux_N):=\frac{1}{N^2}\sum_{1\leq i\neq j\leq N} \g(x_i-x_j),\quad \ux_N:=(x_1,\dotsc,x_N)\in(\T^d)^N. \]
More generally, it is an example of a mean-field interacting diffusion process. These arise in many pure and applied settings where particles or individuals interact pairwise with each other: they describe the dynamics of charged gases~\cite{vlasov_vibrational_1968}, the eigenvalues of random matrices~\cite{dyson_brownianmotion_2004,anderson_introduction_2009}, vortices in viscous fluids~\cite{helmholtz_lxiii_1867,osada_stochastic_1985}, the collective motion of animals or bacteria~\cite{perthame_transport_2007,fournier_stochastic_2017}, and scaling limits for neural networks~\cite{mei_mean_2018,rotskoff_trainability_2022}. Systems with Riesz interactions are particularly interesting, encompassing the first four examples above.

The study of Riesz-type diffusive flows has largely been concerned with showing that they satisfy~\textit{mean-field limits}~\cite{schochet_point-vortex_1996,hauray_wasserstein_2009,carrillo_mass-transportation_2012,carrillo_derivation_2014,berman_propagation_2019}. That is, proving that if the empirical measures of the initial configurations $\mu_N^0$ converge to $\gamma$ as $N\rightarrow\infty$, then $\mu_N^t$ converges to $\mu^t$, the deterministic solution of the McKean-Vlasov equation
\begin{equation}\label{eq:mve}
\begin{cases}
\partial_t\mu^t-\sigma\Delta\mu^t-\nabla\cdot(\mu^t\nabla\g*\mu^t)=0,\\
\mu^0=\gamma,
\end{cases}
\end{equation}
at all future times $t>0$. Mean-field convergence is thus analogous to a law of large numbers: the random empirical measures $\mu_N^t$ converge to a deterministic object $\mu^t$ as the total number of particles goes to infinity. Although there is a vast body of work on the mean-field convergence of interacting diffusions~\cite{mckean_propagation_1967,dobrushin_vlasov_1979,sznitman_topics_1991,jabin_mean_2016,jabin_quantitative_2018,guillin_systems_2023}, it is only relatively recently that mean-field convergence was proven for the full range ($\s\in[0,d))$ of repulsive Riesz interactions.

The main technical innovation that allowed this was the introduction of the \textit{modulated energy} in~\cite{duerinckx_mean-field_2016,serfaty_mean_2020}. Defined for a particle configuration $\ux_N\in(\T^d)^N$ and a probability measure $\mu\in L^\infty(\T^d)$ by
\begin{equation}\label{eq:modulated-energy}
    F_N(\ux_N,\mu):=\int_{(\T^d)^2\setminus\Delta}\g(x-y)\,\diff \bigg(\frac{1}{N}\sum_{i=1}^N\delta_{x_i}-\mu\bigg)^{\otimes 2}(x,y),\quad \Delta:=\big\{(x,y)\in(\T^d)^2\mid x=y\big\},
\end{equation}the modulated energy acts as a pseudo-distance between the empirical measure $\mu_N:=\frac{1}{N}\sum_{i=1}^N \delta_{x_i}$ and $\mu$. In particular, if $\ux_N$ are a sequence of particle configurations so that $\lim_{N\rightarrow \infty}F_N(\ux_N,\mu)= 0$, then the associated empirical measures $\mu_N$ converge to $\mu$ weakly and $F_N$ is asymptotically positive in that there exists $C,\beta>0$ depending on $d$ and $\s$ so that
\begin{equation}\label{eq:modulated_energy_positivity}
F_N(\ux_N,\mu)\geq-C\|\mu\|_{L^\infty(\T^d)}N^{-\beta}.
\end{equation}

In~\cite{serfaty_mean_2020}, given some regularity conditions on the mean-field limit, the modulated energy was used to prove mean-field convergence for Coulomb and super-Coulomb ($\s\geq d-2$) repulsive interactions on $\R^d$ in general dimensions for noiseless systems ($\sigma=0)$ using a Gr\"onwall argument applied to $F_N(\ux_N^t,\mu^t)$. This strategy has proven robust: it has been successfully adapted to handle more general conditions on the mean-field limit~\cite{rosenzweig_mean_2022,rosenzweig_coulomb_2022}, the entire range of repulsive Riesz interactions~\cite{nguyen_mean-field_2022}, and global-in-time mean-field convergence for sub-Coulomb diffusive flows ($\s<d-2$ and $\sigma>0)$~\cite{rosenzweig_global--time_2023}. The related~\textit{modulated free energy}~\cite{bresch_modulated_2019,bresch_mean-field_2019}, which combines the modulated energy with relative entropy, has also been used to show mean-field convergence for a wide class of singularly interacting diffusive flows including logarithmic attraction~\cite{bresch_mean_2023,de_courcel_attractive_2023} and Coulomb/super-Coulomb repulsion~\cite{de_courcel_sharp_2023}.

Despite these recent advances in the mean-field convergence theory of Riesz flows, as well as some forthcoming results showing Gaussian fluctuations around the mean-field limit~\cite{huang_fluctuations_2024}, Riesz flows are only known to satisfy LDPs in a few specific regimes of $d$, $\s$ and $\sigma$.  The first result was~\cite{duvillard_large_2001} where an LDP upper bound and a partial lower bound were shown in the setting of Dyson Brownian motion (DBM) or particles in $\R$ interacting via logarithmic repulsion $(\s=0)$ with $N$-dependent diffusion parameters $\sigma=(\beta N^{-1}).$ Soon after this was expanded to a full LDP~\cite{guionnet_large_2002,guionnet_addendum_2004}, and recently LDPs have also been shown for Dyson Bessel processes~\cite{guionnet_asymptotics_2023}. Outside of random matrix theory, in~\cite{fontbona_uniqueness_2004}, an LDP upper bound and a partial lower bound were established for repulsive systems on $\R$ with $\s=0$ and non-vanishing noise. LDPs were also shown for stochastic interacting vortex systems corresponding to conservative dynamics and $\s=0$ on the two-dimensional torus~\cite{chen_sample-path_2022}. Notably, the interaction potential is logarithmic and the dimension is less or equal to two in all of these regimes.

All of the above papers use \textit{hydrodynamics techniques} as first proposed in~\cite{kipnis_large_1990}. This method avoids some complications that arise when handling singular interactions as the empirical measure of the processes are tested against smooth functions, effectively smoothing the interaction potential. In contrast, most other techniques for establishing LDPs for interacting diffusions require very regular interactions such as methods using the Cameron-Martin-Girsanov theorem~\cite{dawsont_large_1987}, Hamilton-Jacobi theory~\cite[Section 13.3]{feng_large_2006}, ideas from stochastic optimal control theory~\cite{budhiraja_large_2012}, or the contraction principle~\cite{coghi_pathwise_2020}. Laplace's method as introduced in~\cite{ben_arous_methode_1990} to derive LDPs for smoothly interacting diffusions has been successfully extended to give LDPs for interactions strictly less singular than logarithmic ($\s=0$)~\cite{hoeksema_large_2024}, but seems intractable for Riesz interactions.

Until now, no LDP-type results for Riesz interactions in dimensions greater than two or $\s>0$ have been shown.

\subsection{Setup}
We associate $\T^d$ with $[-\frac{1}{2},\frac{1}{2}]^d$ under periodic boundary conditions. We abuse notation by letting $|x-y|$ denote the usual distance between any two points $x$ and $y$ on the torus. Thus by $|x|$, we mean the distance from $x$ to 0.

We endow the space of probability measures $\Pc(\T^d)$ with the topology of weak convergence metrized by the Wasserstein-1 distance
\begin{equation}\label{eq:weak_metric} 
d(\mu,\nu):=\sup_{\|\nabla\psi\|_{L^\infty(\T^d)}\leq 1} \bigg|\int_{\T^d} \psi(x)\,\diff(\mu-\nu)(x)\bigg|.
\end{equation}
We then endow $C([0,T],\Pc(\T^d))$ with the uniform topology.

The periodic Riesz potential $\g$ associated to parameter $\s\in[0,d)$ is taken to be the unique zero-mean solution to
\begin{align}\label{eq:Riesz}
|\nabla|^{d-\s}\g=\c_{d,\s}(\delta_0-1),\qquad \c_{d,\s}:=\begin{cases}
    \frac{4^{(d-\s)/2}\Gamma((d-\s)/2)\pi^{d/2}}{\Gamma(s/2)}&\s\in(0,d),\\
    \frac{\Gamma(d/2)(4\pi)^{d/2}}{2}&\s=0,
\end{cases}
\end{align}
where $|\nabla|$ denotes the operator $(-\Delta)^{1/2}$ with Fourier multiplier $2\pi|k|$.
As shown in~\cite{hardin_next_2017}, $\g$ is smooth away from $0$ and
\begin{equation}\label{eq:periodic-correction}
    \g(x)-\Big(|x|^{-\s}\indc_{\s>0}-\log|x|\indc_{\s=0}\Big)\in C^\infty\big(B_{1/4}(0)\big).
\end{equation}
That is, $\g$ is the Euclidean Riesz potential near the origin up to a smooth correction.

The \textit{Riesz energy} and \textit{Riesz enstrophy} of a positive measure $\mu$ are respectively defined by
\[\Ec(\mu):=\int_{(\T^d)^2} \g(x-y)\,\diff \mu^{\otimes2}(x,y)\quad\text{and}\quad\Dc(\mu):=\int_{(\T^d)^2} (-\Delta)\g(x-y)\,\diff \mu^{\otimes2}(x,y).\]
Then, setting
\[
\|f\|_{\dot{H}^\alpha(\T^d)}^2:=\sum_{k\in\Z^d:k\neq 0}(2\pi|k|)^{2\alpha}|\hat{f}(k)|^2,\qquad\alpha\in\R
\]
it holds that
\[\Ec(\mu)=\c_{d,\s}\|\mu\|_{\H(\T^d)}^2\quad\text{and}\quad \Dc(\mu)=\c_{d,\s}\|\mu\|_{\dot H^{1+\frac{\s-d}{2}}(\T^d)}^2\]
by the Plancherel theorem. For a trajectory of probability measures $\mu\in C([0,T],\Pc(\T^d))$ we then let
\[Q(\mu):=\sup_{t\in[0,T]}\bigg\{\Ec(\mu^t)+2\sigma\int_0^t\Dc(\mu^\tau)\,\diff \tau\bigg\}.\]
This corresponds to the natural energy inequality for the McKean-Vlasov equation~\eqref{eq:mve}.

Motivated by the fact that when $\mu$ is a smooth density
\[\int_{\T^d}\mu(x)\nabla\g*\mu(x)\cdot\psi(x)\,\diff x=\frac{1}{2}\int_{(\T^d)^2}(\psi(x)-\psi(y))\cdot\nabla\g(x-y) \mu(x)\mu(y)\,\diff x\,\diff y\]
by symmetrizing, whenever $\Ec(\mu)<\infty$ we abuse notation and let $\mu\nabla\g*\mu$ denote the distribution
\begin{equation}\label{def:mu_grad_mu}
\psi\mapsto \frac{1}{2}\int_{(\T^d)^2}(\psi(x)-\psi(y))\cdot\nabla\g(x-y)\,\diff \mu^{\otimes 2}(x,y).
\end{equation}
This is well-defined since~\eqref{eq:periodic-correction} implies that there exists $C(d,\s)>0$ so that
\[|(\psi(x)-\psi(y))\cdot\nabla\g(x-y)|\leq  C\|\nabla\psi\|_{L^\infty(\T^d)}(\g(x-y)+C),\]
for all $x,y\in\T^d$.

If $Q(\mu)<\infty$ and $\phi\in C^\infty([0,T]\times\T^d)$ we set
\begin{align*}
S(\mu,\phi):=&\langle\mu^T,\phi^T\rangle-\langle\mu^0,\phi^0 \rangle  -\int_0^T \langle\mu^t,\partial_t\phi^t+\sigma\Delta\phi^t \rangle\,\diff t-\sigma\int_0^T\langle\mu^t, |\nabla\phi^t|^2\rangle\,\diff t
\\\notag&\qquad+\frac{1}{2}\int_0^T\int_{(\T^d)^2} (\nabla\phi^t(x)-\nabla\phi^t(y))\cdot\nabla\g(x-y)\,\diff (\mu^t)^{\otimes 2}(x,y)\,\diff t.
\end{align*}
The last term on the right-hand side is well-defined for the same reason as~\eqref{def:mu_grad_mu}.

\subsection{Main results}

Throughout we assume that $\gamma$ is in $\Pc(\T^d)\cap L^\infty(\T^d)$ and that the initial conditions of~\eqref{eq:SDE} are well-prepared in the following way.
  
\begin{assumption}\label{cond:initial-convergence}
\[\lim_{N\rightarrow\infty} F_N(\ux_N^0,\gamma)=0,\]
 where $F_N$ is the modulated energy defined in~\eqref{eq:modulated-energy}.
\end{assumption}

\begin{remark} Assumption~\ref{cond:initial-convergence} holds if and only if $\mu_N^0\rightarrow \gamma$ in $\Pc(\T^d)$ and $H_N(\ux_N^0)\rightarrow \Ec(\gamma)$ as $N\rightarrow \infty.$ As we do not use this anywhere we omit the proof.
\end{remark}

Given this assumption, we show large deviation estimates for the empirical trajectories of~\eqref{eq:SDE} with respect to the (lower semi-continuous) rate function
\begin{equation}\label{def:rate_function}
I_\gamma(\mu):=\begin{cases}
\sup\limits_{\phi\in C^\infty([0,T]\times\T^d)} S(\mu,\phi)\vee\frac{1}{2\sigma}\big(Q(\mu)-\mathcal{E}(\gamma)\big)& Q(\mu)<\infty\text{ and }\mu^0=\gamma,\\
+\infty&\text{otherwise},
\end{cases}
\end{equation}
with local LDP lower bounds on the dense subset
\begin{equation}\label{def:dense_set}
\mathscr{A}:=\bigg\{\mu\in C([0,T],\Pc(\T^d))\ \Big|\ \int_0^T \big\||\nabla|^{\frac{1}{2}+\frac{\s-d}{2}}\mu^t\big\|_{L^{\frac{6d}{3d-\s-1}}(\T^d)}^3\,\diff t<\infty\bigg\}.    
\end{equation}

\begin{theorem}\label{thm:LDP} For any $\gamma\in \Pc(\T^d)\cap L^\infty(\T^d)$ and $\ux_N^0$ satisfying Assumption~\ref{cond:initial-convergence} it holds that:
        \begin{enumerate}
        \item\label{item:goodness} $I_\gamma$ has compact sublevel sets.
        
            \item\label{item:LDP-upper-bound} For all closed $F\subset C([0,T],\Pc(\T^d))$
        \begin{equation*}
        \limsup_{N\rightarrow\infty}\frac{1}{N}\log\P(\mu_N\in F)\leq -\inf_{\mu\in F} I_\gamma(\mu).
        \end{equation*}
        \item\label{item:LDP-lower-bound} For all open $O\subset C([0,T],\Pc(\T^d))$
        \begin{equation*}
        \liminf_{N\rightarrow\infty}\frac{1}{N}\log\P(\mu_N\in O)\geq \begin{cases}
            -\inf\limits_{\mu\in O\cap\mathscr{A}} I_\gamma(\mu)& \s>0,\\
            -\inf\limits_{\mu\in O} I_\gamma(\mu)& \s=0.
        \end{cases}
        \end{equation*}
        \end{enumerate}
\end{theorem}

\begin{remark}\label{rem:A_equivalence}
We will see in Subsection~\ref{subsec:lower-bound} that if $\mu\in\A$ and $Q(\mu)<\infty$ then
\[ I_\gamma(\mu)=\sup\limits_{\phi\in C^\infty([0,T]\times\T^d)} S(\mu,\phi).\]
This implies that Item~\ref{item:LDP-upper-bound} and~Item~\ref{item:LDP-lower-bound} in Theorem~\ref{thm:LDP} hold with $I_\gamma$ replaced by
\[\tilde{I}_\gamma(\mu):=\begin{cases}
  \sup\limits_{\phi\in C^\infty([0,T]\times\T^d)} S(\mu,\phi)& Q(\mu)<\infty\text{ and }\mu^0=\gamma,
  \\+\infty&\text{otherwise}.
\end{cases}\]
We state Theorem~\ref{thm:LDP} with $I_\gamma$ as we cannot show that $\tilde I_\gamma$ is lower semi-continuous when $\s>0.$ In particular, there could exist a sequence of measure trajectories $\mu_k$ converging to $\mu$ in $C([0,T],\Pc(\T^d))$ so that $Q(\mu)=\infty$ but
\[\limsup_{k\rightarrow\infty}\sup\limits_{\phi\in C^\infty([0,T]\times\T^d)} S(\mu_k,\phi)<\infty.\]
When $\s=0$, since $\mu\in\A$ whenever $I_\gamma(\mu)<\infty$, an LDP holds with good rate function $\tilde I_\gamma.$
\end{remark}

\begin{remark}\label{rem:perturbed_representation} If $I_\gamma(\mu)<\infty$ then there exists a vector field $b$ so that $\mu$ is a weak solution to
to the perturbed McKean-Vlasov equation
\begin{equation}\label{eq:mve+b}
\begin{cases}
    \partial_t\mu^t -\sigma\Delta\mu^t-\nabla\cdot(\mu^t\nabla\g*\mu^t) =-\nabla\cdot(b^t\mu^t),
    \\\mu^0=\gamma,
\end{cases}
\end{equation}
and
\[\sup_{\phi\in C^\infty([0,T]\times\T^d)}S(\mu,\phi)=\frac{1}{4\sigma}\int_0^T \int_{\T^d}|b^t(x)|^2\,\diff\mu^t(x)\,\diff t,\]
where $\mu^t\nabla\g*\mu^t$ in~\eqref{eq:mve+b} is defined by~\eqref{def:mu_grad_mu}.
This follows by a simple argument that uses the Riesz representation theorem and the fact that the fourth term in $S(\mu,\phi)$ is quadratic in $\phi.$

As a consequence, 
$I_\gamma(\mu)=0$ if and only if $Q(\mu)<\infty$, $\mu^0=\gamma$, and $\mu$ is a solution to the non-perturbed McKean-Vlasov equation~\eqref{eq:mve}. Also, as in~\cite{dawsont_large_1987} and~\cite{chen_sample-path_2022}, it holds that
\[\sup\limits_{\phi\in C^\infty([0,T]\times\T^d)} S(\mu,\phi)=\frac{1}{4\sigma}\int_0^T\big\|\partial_t\mu^t-\sigma\Delta\mu^t-\nabla\cdot(\mu^t\nabla\g*\mu^t)\big\|_{-1,\mu^t}^2\,\diff t,\]
where for any probability measure $\nu$ and distribution $T$
\[\|T\|_{-1,\nu}^2:=\sup_{\psi\in C^\infty(\T^d)}\Bigg\{2\langle T,\psi\rangle-\int_{\T^d} |\nabla\psi|^2\,\diff\nu\Bigg\}.\]
\end{remark}

\begin{theorem}\label{thm:local-LDP} Given Assumption~\ref{cond:initial-convergence}, for all $\mu\in C([0,T],\Pc(\T^d))\cap L^\infty([0,T],L^\infty(\T^d))$
\begin{align*}
\lim_{\eps\rightarrow0}\limsup_{N\rightarrow\infty}\frac{1}{N}\log\P\bigg(\sup_{t\in[0,T]}F_N(\ux_N^t,\mu^t)<\eps\bigg)&\leq-\sup_{\phi\in C^\infty([0,T]\times\T^d)}S(\mu,\phi)
    \\&\leq\lim_{\eps\rightarrow 0}\liminf_{N\rightarrow\infty}\frac{1}{N}\log\P\bigg(\sup_{t\in[0,T]}F_N(\ux_N^t,\mu^t)<\eps\bigg).
\end{align*}
Due to the obvious bound of $\liminf\leq \limsup$, the above inequalities are equalities.
\end{theorem}

\begin{remark} The modulated energy $F_N(\ux_N^t,\mu^t)$ should be thought of as the renormalized version of the $\H(\T^d)$ distance between $\mu_N^t$ and $\mu^t$. Theorem~\ref{thm:local-LDP} thus shows that a local LDP holds in a stronger topology than $C([0,T],\Pc(\T^d))$.
\end{remark}

\subsection{Overview}

Our strategy generally follows that proposed in~\cite{kipnis_large_1990}: the upper bound is derived using exponential martingales and the lower bound is derived by considering appropriate regular perturbations of the system. This scheme is essentially used in~\cite{duvillard_large_2001,guionnet_large_2002,guionnet_addendum_2004,fontbona_uniqueness_2004,chen_sample-path_2022} to derive LDPs for Riesz interacting systems. As opposed to those papers, to overcome the difficulties introduced when the interaction potential is not logarithmic $(\s>0)$, we must use tools developed in the study of the mean-field convergence of Riesz flows.

We first summarize this general scheme before elaborating on where this work diverges from previous papers.

\begin{itemize}
\item \textit{Upper bound:} The empirical trajectories $\mu_N$ associated to~\eqref{eq:SDE} are first shown to be exponentially tight in $C([0,T],\Pc(\T^d)).$ It thus suffices to prove a weak large deviation upper bound. This is achieved by testing $\mu_N^t$ against a smooth functions $\phi\in C^\infty([0,T]\times\T^d).$ For fixed $\phi$ by applying It\^o's formula to $\langle \mu_N^t,\phi^t\rangle$ and rearranging one finds that there exists a continuous martingale $M^t$ adapted to the filtration of the noise with bounded quadratic variation so that 
\[S_N(\ux_N,\phi)= M^T-\frac{N}{2}\langle M\rangle^T\]
where
\begin{align*}
S_N(\ux_N,\phi)&:=\langle\mu_N^T,\phi^T\rangle-\langle\mu_N^0,\phi^0 \rangle  -\int_0^T \langle\mu_N^t,\partial_t\phi^t+\sigma\Delta\phi^t \rangle\,\diff t-\sigma\int_0^T\langle\mu_N^t, |\nabla\phi^t|^2\rangle\,\diff t
\\\notag&\qquad+\frac{1}{2}\int_0^T\int_{(\T^d)^2\setminus\Delta}(\nabla\phi^t(x)-\nabla\phi^t(y))\cdot\nabla\g(x-y)\,\diff (\mu_N^t)^{\otimes 2}(x,y)\,\diff t.
\end{align*}
The advantage of this is that $\exp(N M^t-\frac{N^2}{2}\langle M\rangle^t)$ is also a martingale~\cite[Proposition 5.11]{le_gall_measure_2022}, hence has a constant expectation (in this case equal to 1). Combined with Chebyshev's inequality this gives an exponential bound on the probability $\mu_N$ is near some measure trajectory $\mu$, namely that
\[\P\big(\mu_N\in B_\eps(\mu)\big)\leq \exp\Big({-N}\inf_{\mu_N\in B_\eps(\mu)}S_N(\ux_N,\phi)\Big).\]
If $S_N(\ux_N,\phi)\rightarrow S(\mu,\phi)$ whenever $\mu_N\rightarrow \mu$ in $C([0,T],\Pc(\T^d))$ this then implies (after optimizing over $\phi$) that
\[\lim_{\eps\rightarrow0}\limsup_{N\rightarrow\infty}\frac{1}{N}\log\P\big(\mu_N\in B_\eps(\mu)\big)\leq -\sup_{\phi\in C^\infty([0,T],\T^d)}S(\mu,\phi).\]

\item \textit{Lower bound:} Here one takes advantage of the fact that when $I_\gamma(\mu)<\infty$,
$\mu$ must be a weak solution to the perturbed McKean-Vlasov solution~\eqref{eq:mve+b} as noted in Remark~\ref{rem:perturbed_representation}. This is the formal mean-field limit of the particle systems
\begin{equation}\label{eq:SDE+b}
\begin{cases}
\diff x_i^t=-\frac{1}{N}\sum\limits_{1\leq j\leq N:j\neq i} \nabla\g(x_i^t-x_j^t)\,\diff t+b^t(x_i^t)\,\diff t +\sqrt{2\sigma}\,\diff w^t_i,\\
x_i^t|_{t=0}=x_i^0,
\end{cases}\qquad i\in\{1,\dotsc,N\}
,\end{equation}
and when $\mu$ and $b$ are sufficiently regular the empirical trajectories of~\eqref{eq:SDE+b} almost surely converge to the solution of~\eqref{eq:mve+b} in $C([0,T],\Pc(\T^d)).$ Since the Radon-Nikodym derivative between the law of the non-perturbed system~\eqref{eq:SDE} and the law of the perturbed system~\eqref{eq:SDE+b} is equal to
\begin{equation}\label{eq:girsanov_transform}
\exp\bigg(\frac{1}{\sqrt{2\sigma}}\sum_{i=1}^N\int_0^T b^t(x_i^t)\cdot \diff w_i^t-\frac{1}{4\sigma}\sum_{i=1}^N\int_0^T |b^t(x_i^t)|^2\,\diff t\bigg)
\end{equation}
by the Girsanov theorem, the mean-field convergence of the perturbed systems gives a local lower LDP bound of the form
\[\lim_{\eps\rightarrow 0}\liminf_{N\rightarrow\infty}\frac{1}{N}\log\P\big(\mu_N\in B_\eps(\mu)\big)\geq -\frac{1}{4\sigma}\int_0^T \int_{\T^d} |b^t(x)|^2\,\diff \mu^t(x)\,\diff t=-\sup_{\phi\in C^\infty([0,T]\times\T^d)}S(\mu,\phi)\]
via a tilting argument. The lower bound for less regular $\mu$ is then recovered via approximation.
\end{itemize}

The main stumbling block in completing this approach when $\g$ is singular is that the convergence of $\mu_N\rightarrow \mu$ in $C([0,T],\Pc(\T^d))$ does not guarantee
that $S_N(\ux_N,\phi)\rightarrow S(\mu,\phi).$ In particular, letting
\[\K_{\psi}(x,y):=(\psi(x)-\psi(y))\cdot\nabla\g(x-y),\]
it is problematic to show that for arbitrary $\phi\in C^\infty([0,T]\times \T^d)$
\begin{equation}\label{eq:convergence_overview}\int_0^T\int_{(\T^d)^2\setminus\Delta}\K_{\nabla\phi^t}(x,y)\,\diff (\mu_N^t)^{\otimes 2}(x,y)\,\diff t\rightarrow \int_0^T\int_{(\T^d)^2}\K_{\nabla\phi^t}(x,y)\,\diff (\mu^t)^{\otimes 2}(x,y)\,\diff t.
\end{equation}

When $\s=0$, this convergence is almost immediate for non-atomic measure trajectories. $\K_\psi(x,y)$
is bounded and continuous away from the diagonal whenever $\psi$ is Lipschitz, thus using a straightforward argument (which is essentially Delort's theorem~\cite[Lemma 6.3.1]{chemin_perfect_1998})~\eqref{eq:convergence_overview} holds whenever $\mu_N\rightarrow \mu$ in $C([0,T],\Pc(\T^d))$ and $\mu^t$ is non-atomic for almost every $t$.

When $\s>0$, $\K_\psi(x,y)$ can be unbounded, and we instead use a so-called \textit{commutator estimate} to control the difference between the two sides of~\eqref{eq:convergence_overview}. We also use some \textit{\textit{a priori}} bounds on the particle trajectories $\ux_N$: we introduce a family of auxiliary functions $Q_N$ for which we can give good exponential bounds on the probability that $Q_N(\ux_N)$ is large. These functions motivate the definition of $Q$ and are necessary for showing the convergence of $S_N$ to $S.$ Although $Q$ is similar to the auxiliary function introduced in~\cite{chen_sample-path_2022}, we use the quantitative control on all terms in $Q_N$ to complete the upper bound. This is in contrast to~\cite{chen_sample-path_2022} where the \textit{a priori} bounds are used to reduce the class of measure trajectories for which the local LDP upper bounds need to be proved.

We also use a commutator estimate and the modulated energy to prove the mean-field convergence of the perturbed discrete flows~\eqref{eq:SDE+b} to the perturbed McKean-Vlasov equation~\eqref{eq:mve+b}. This is in contrast to an argument using compactness and the unique existence of a weak solution to~\eqref{eq:mve+b} as used in previous papers. By doing this we avoid some difficulties in giving a natural regularity class for which weak solutions to~\eqref{eq:mve+b} are unique and we prove a mean-field limit in a stronger topology than $C([0,T],\Pc(\T^d)),$ allowing us to prove Theorem~\ref{thm:local-LDP}.

Finally, in~\cite{chen_sample-path_2022}, the auxiliary function is additionally used in the approximation argument that extends the lower bound for regular measure trajectories to arbitrary measure trajectories. Due to numerological issues, we can only do this when $\s=0$, and can generally only extend the lower bound to $\mu$ that are in $\A.$

Below we elaborate on these main points before continuing to the body of the paper.

\subsubsection{Auxiliary functionals and \textit{a priori} bounds}
The definition of $Q$ is motivated by the energy dissipation structure of~\eqref{eq:SDE}. In particular, letting 
\[
Q_N(\ux_N):=\sup_{t\in[0,T]} \bigg\{H_N(\ux_N^t)+2\sigma\int_0^tD_N(\ux_N^\tau)\,\diff \tau\bigg\},\]
where
\[D_N(\ux_N):=\frac{1}{N^2}\sum_{1\leq i\neq j\leq N}(-\Delta)\g(x_i-x_j)\]
is the discrete Riesz enstrophy, given Assumption~\ref{cond:initial-convergence} it holds that 
\begin{equation}\label{eq:overview_energy_bounds}\limsup_{N\rightarrow\infty}\frac{1}{N}\log\P\Big(Q_N(\ux_N)\geq L\Big)\leq-\frac{1}{2\sigma}\big(L-\Ec(\gamma)\big).
\end{equation}
This essentially follows by applying It\^o's formula to $H_N(\ux_N^t)$ and crucially uses that $\s<d-2$ so that the It\^o correction term $D_N(\ux_N^t)$ can be bounded below.
 
 As a consequence of~\eqref{eq:overview_energy_bounds} (and that $Q_N$ satisfy a $\Gamma$-limit lower bound with respect to $Q$), we can restrict our attention to measure trajectories such that $Q(\mu)<\infty$ (for which $S(\mu,\phi)$ is well-defined) and in the upper bound we only need to show convergence of $S_N(\ux_N,\phi)$ to $S(\mu,\phi)$ when both $\mu_N\rightarrow\mu$ in $C([0,T],\Pc(\T^d))$ and there exists $L>0$ so that $Q_N(\ux_N)\leq L$ for all $N$.

\subsubsection{\texorpdfstring{$S_N$}{} convergence}
To show the convergence of the last term in $S_N(\ux_N,\phi)$ to the last term in $S(\mu,\phi)$ we use the following proposition.
\begin{proposition}\label{prop:renormalized_commutator_estimate} There exists $C(d,\s),\beta(d,\s)>0$ so that for every $\ux_N\in(\T^d)^N$ pairwise distinct, $\mu\in \Pc(\T^d)\cap L^\infty(\T^d)$, and Lipschitz vector field $\psi$
\begin{equation}\label{eq:old-commutator-estimate}
\Bigg|\int_{(\T^d)^2\setminus\Delta} \K_\psi(x,y)\,\diff(\mu_N-\mu)^{\otimes2}(x,y) \Bigg|\leq CA_\psi\bigg( F_N(\ux_N,\mu)+C\|\mu\|_{L^\infty(\T^d)}N^{-\beta}\bigg),
\end{equation}
and
\begin{align}\label{eq:renormalized_commutator_estimate}
    &\bigg|\int_{(\T^d)^2\setminus\Delta} \K_\psi(x,y)\,\diff (\mu_N^{\otimes 2}-\mu^{\otimes 2})(x,y)\bigg|
    \\\notag&\quad\leq C A_\psi\Big(F_N(\ux_N,\mu)+C\|\mu\|_{L^\infty(\T^d)}N^{-\beta}\Big)^{\frac{1}{2}}\Big(H_N(\ux_N)+\|\mu\|_{\H(\T^d)}^2+C\|\mu\|_{L^\infty(\T^d)}N^{-\beta}+1\Big)^{\frac{1}{2}},
\end{align}
where $A_\psi:=\|\nabla \psi\|_{L^\infty(\T^d)}+\big\||\nabla|^{\frac{d-\s}{2}}\psi\big\|_{L^{\frac{2d}{d-2-\s}}(\T^d)}$.
\end{proposition}

The bound~\eqref{eq:old-commutator-estimate} is a version of the commutator estimate~\cite[Proposition 1.1]{serfaty_mean_2020} adapted to sub-Coulomb periodic Riesz potentials. In~\cite{serfaty_mean_2020},~\eqref{eq:old-commutator-estimate} is used to bound a term in the derivative of $F_N(\ux_N^t,\mu^t)$ by itself and complete the Gr\"onwall argument. The second inequality~\eqref{eq:renormalized_commutator_estimate} is an easy consequence of~\eqref{eq:old-commutator-estimate}. As both estimates are proved using standard renormalization ideas, we give their proofs in Appendix~\ref{appendix}.

The second inequality in Proposition~\ref{prop:renormalized_commutator_estimate} implies that if $\mu\in L^\infty([0,T],L^\infty(\T^d))$, the discrete energies $H_N(\ux_N^t)$ are uniformly bounded over $N$ and $t$, and
\begin{equation}\label{eq:modulated-energy-convergence-overview}
\int_0^T F_N(\ux_N^t,\mu^t)\,\diff
t\rightarrow 0,
\end{equation}
then the last term in $S(\ux_N,\phi)$ converges to the last term in $S(\mu,\phi)$ for any smooth $\phi$. The convergence~\eqref{eq:modulated-energy-convergence-overview} holds given that $\mu_N\rightarrow \mu$ in $C([0,T],\Pc(\T^d))$ and $\int_0^T D_N(\ux_N^t)\,\diff t$ are uniformly bounded by interpolating between the Wasserstein-1 metric and the discrete enstrophy $D_N$. As uniform bounds on $Q_N$ imply uniform bounds on both the discrete energies and discrete enstrophies, we find that if $\mu_N\rightarrow\mu$ in  $C([0,T],\Pc(\T^d))$ and $Q_N(\ux_N)\leq L$, then
$S_N(\ux_N,\phi)$ converges to $S(\mu,\phi).$ It is critical that the prefactor before the enstrophy in the \textit{a priori} bounds is uniformly bounded away from zero, thus we cannot handle vanishing noise as in Dyson Brownian motion. When $\mu$ is not in $L^\infty([0,T],L^\infty(\T^d))$, we adapt this argument by appropriately mollifying $\mu$.

\subsubsection{Mean-field limit for regular perturbations}

We prove that the empirical trajectories of~\eqref{eq:SDE+b} converge to~\eqref{eq:mve+b} using a modulated energy argument. This uses~\eqref{eq:old-commutator-estimate} and is similar to the proof of mean-field convergence in~\cite{rosenzweig_global--time_2023}, although we combine Gr\"onwall’s inequality with Doob’s martingale inequality to give almost sure uniform bounds in time. That is, when $\mu$ with associated vector field $b$ are sufficiently regular and Assumption~\ref{cond:initial-convergence} holds we prove that
\[\lim_{N\rightarrow\infty}\sup_{t\in[0,T]} F_N(\ux_N^t,\mu^t)=0\quad \text{almost surely},\]
as opposed to
\[\lim_{N\rightarrow\infty}\sup_{t\geq 0}\E\Big[F_N(\ux_N^t,\mu^t)\Big]=0,\]
as in~\cite{rosenzweig_global--time_2023}. The argument again
uses the restriction that $\s<d-2$ to bound an It\^o correction term involving $(-\Delta)\g$ from below. This mean-field limit with respect to the modulated energy is critical for showing Theorem~\ref{thm:local-LDP}.

\subsubsection{Approximating sequences}

After using the mean-field limit and tilting to establish a local lower bound for regular perturbations of the system we expand the class of measure trajectories for which the local lower bounds hold to $\mu\in\A$. In particular, we show that if $\mu\in\A$ and $I_\gamma(\mu)<\infty$ then there exists a sequence of measure trajectories $\nu_\eps$ converging to $\mu$ in $C([0,T],\Pc(\T^d))$ as $\eps\rightarrow 0$ so that the local lower bound holds for all $\nu_\eps$ and 
\begin{equation}\label{eq:rate_funciton_convergence}
    \lim_{\eps\rightarrow 0} I_\gamma(\nu_\eps)=I_\gamma(\mu).
\end{equation}
This allows us to recover the lower bound for $\mu$. These $\nu_\eps$ are essentially space mollifications of $\mu$ by the heat kernel $\Phi^\eps$. It is here that we use that we are on the torus as it guarantees that space mollifications of solutions to~\eqref{eq:mve+b} are also solutions~\eqref{eq:mve+b}, but with regular drifts.  

Letting $(f)_\eps$ denote $f*\Phi^\eps$, showing~\eqref{eq:rate_funciton_convergence} reduces to proving that
\begin{equation}\label{eq:intro_commutator}\int_0^T\int_{\T^d}\bigg|\frac{(\mu^t\nabla\g*\mu^t)_\eps}{(\mu^t)_\eps}-\nabla\g*(\mu^t)_\eps\bigg|^2\,\diff(\mu^t)_\eps\,\diff t\rightarrow 0,
\end{equation}
as $\eps\rightarrow 0$. $\A$ is the largest class of measure trajectories for which we can show~\eqref{eq:intro_commutator} since it is the largest class for which we can make sense of $\int_0^T \int|\nabla\g*\mu^t|^2\,\diff\mu^t\,\diff t$ as a pairing of three Sobolev distributions. We verify that when $\s=0$, $\mu\in\A$ whenever $Q(\mu)<\infty$, hence the lower bound holds for all admissible trajectories.

\subsubsection{Layout of paper} In Section~\ref{sec:a_priori_bounds} we show that the SDE~\eqref{eq:SDE} is well-posed and prove the exponential bounds on the probability $Q_N(\ux_N)$ is large. In Section~\ref{sec:continuity_of_S} we show that $S_N(\ux_N,\phi)$ converges to $S(\mu,\phi)$ when $Q_N(\ux_N)$ is uniformly bounded. We also show that $I_\gamma$ is a good rate function. The exponential tightness of $\mu_N$, the upper bound of Theorem~\ref{thm:LDP}, and the first inequality of Theorem~\ref{thm:local-LDP} are given in Section~\ref{sec:upper_bounds}. We prove a quantitative mean-field limit for regular perturbations of~\eqref{eq:SDE} in Section~\ref{sec:mean_field_limit}. The lower bound of Theorem~\ref{thm:LDP} and the second inequality of Theorem~\ref{thm:local-LDP} are in Section~\ref{sec:lower_bounds}. Finally, the proof of Proposition~\ref{prop:renormalized_commutator_estimate} and related estimates are given in Appendix~\ref{appendix}.

\subsection{Notation}
We use the following notation and conventions throughout the rest of the paper.
\begin{itemize}
\item Unless ambiguous, we drop the domain $\T^d$ in spaces and norms.
\item We let $\Mc(\T^d)$ denote the space of signed Borel measures on $\T^d$ with bounded total variation.
\item For the sake of brevity $\Cs^T:=C([0,T],\Pc(\T^d))$ throughout.
\item  As is convention, we allow $C$ to be a large constant that changes line by line. Both $\beta$ and $C$ are always allowed to depend on $d$ and $\s$. For a set of parameters $\Theta$ we let $C_\Theta$ be a large constant depending on $\Theta.$
\end{itemize}
\subsection{Acknowledgements} The author would like to thank their doctoral advisor, Sylvia Serfaty, for their discussions, guidance, and support. They would also like to thank G\'erard Ben Arous, Philip Gaddy, Matthew Rosenzweig, Keefer Rowan, and Ofer Zeitouni for their insightful discussions and helpful suggestions. The author acknowledges the support of the Natural Sciences and Engineering Research Council of Canada (NSERC), [funding reference number PGSD3-557776-2021]. Cette recherche a \'et\'e financ\'ee par le Conseil de recherches en sciences naturelles et en g\'enie du Canada (CRSNG), [num\'ero de r\'ef\'erence PGSD3-557776-2021].

\section{Existence and energy bounds}\label{sec:a_priori_bounds}

In this section, we give exponential bounds on the probability that $Q_N(\ux_N)$ is large. These are \textit{a priori} bounds derived from the gradient flow structure of~\eqref{eq:SDE}. Formally, the argument proceeds by applying It\^o's formula to $H_N(\ux_N^t)$ when $\ux_N^t$ solves~\eqref{eq:SDE} to find that
\begin{align*}&H_N(\ux_N^t)+2\sigma\int_0^tD_N(\ux_N^\tau)\,\diff \tau-H_N(\ux_N^0)
\\&\quad=\frac{2\sqrt{2\sigma}}{N}\sum_{i=1}^N\int_0^t \bigg(\frac{1}{N}\sum_{1\leq j\leq N:j\neq i}\nabla\g(x_i-x_j)\bigg)\cdot\,\diff w_i^\tau-2\int_0^t\frac{1}{N}\sum_{i=1}^N\bigg|\frac{1}{N}\sum_{1\leq j\leq N:j\neq i}\nabla\g(x_i-x_j)\bigg|^2\,\diff\tau.
\end{align*}
The right-hand side is equal to $M^t-\frac{N}{4\sigma}\langle M\rangle^t$ for some martingale $M^t$, thus we can give exponential bounds on the probability the left-hand side is ever large on $[0,T]$ using Doob's martingale inequality.

This formal argument is complicated by the fact that $\g$ is singular, and we instead consider systems interacting via smooth truncations of $\g$. Our discussion is guided by the proofs in~\cite[Lemma 4.33]{anderson_introduction_2009} and~\cite{rosenzweig_global--time_2023} that the SDE~\eqref{eq:SDE} is well-posed. We take advantage of the fact that systems of particles interacting via a smooth truncation of $\g$ that is equal to $\g(x)$ when $|x|\geq \delta$ agree with the solution of~\eqref{eq:SDE} as long as all the particles remain more than a distance $\delta$ from each other. Since the energy of these truncated systems also gets large when any two particles get close, their natural \textit{a priori} bounds show that as $\delta\rightarrow 0$ the probability any two particles are closer than $\delta$ also goes to zero and thus the untruncated system is the limit of the truncated systems.

The following exponential version of Doob's martingale inequality is crucial for the desired exponential bounds and is also important for proving the mean-field limit later in the paper. The proof is standard, see~\cite[Lemma 3.6]{chen_sample-path_2022}.

\begin{lemma}\label{lemma:exponential-doobs} Let $M^t$ be a positive continuous martingale. Then for any $L\in\R$
\[\P\bigg(\sup_{t\in[0,T]}\log M^t\geq L\bigg)\leq \E[M^0]e^{-L}.\]
\end{lemma}

We also need the following proposition to show that Assumption~\ref{cond:initial-convergence} implies that $H_N(\ux_N^0)\rightarrow \Ec(\gamma)$ as $N\rightarrow \infty$.

\begin{lemma}\label{lem:renormalized_Holder} There exists $C,\beta>0$ so that for any $\ux_N\in (\T^d)^N$ pairwise distinct and $\mu,\nu\in \Pc(\T^d)\cap L^\infty(\T^d)$
\begin{align*}
&\bigg|\int_{(\T^d)^2\setminus\Delta}\g(x-y)\,\diff (\mu_N-\mu)(x)\,\diff (\mu-\nu)(y)\bigg|
\\&\quad\leq C\Big(F_N(\ux_N,\mu)+C\|\mu\|_{L^\infty(\T^d)}N^{-\beta}\Big)\|\mu-\nu\|_{\H(\T^d)}+C\|\mu-\nu\|_{L^\infty(\T^d)}N^{-\beta}.
\end{align*}
\end{lemma}

This is a renormalized version of the Cauchy--Schwarz inequality. We defer the proof to Subsection~\ref{subsec:modulated_energy inequalities} of the Appendix.

As they are important for the LDP lower bound, we show the unique existence of solutions to the perturbed systems~\eqref{eq:SDE+b} when the vector field $b$ is sufficiently regular.

\begin{proposition}\label{prop:energy-estimate} If $b\in L^2([0,T],C^1(\T^d))$ and $\ux_N^0$ are pairwise distinct, then the stochastic differential equation~\eqref{eq:SDE+b} admits a unique strong (and weak) solution so that
$Q_N(\ux_N)$ is almost surely finite. More so, if Assumption~\ref{cond:initial-convergence} holds and $\ux_N$ is the unique solution to~\eqref{eq:SDE}, then
\begin{equation}\label{eq:energy-control}
\limsup_{N\rightarrow\infty}\frac{1}{N}\log\P\Big(Q_N(\ux_N)\geq L\Big)\leq-\frac{1}{2\sigma}\big(L-\Ec(\gamma)\big).
\end{equation}
\end{proposition}
\begin{proof}
Let $\chi$ be a smooth function so that $0\leq \chi\leq 1$ and
\[
\chi(x)=\begin{cases}
0&|x|\geq 1,\\
1&|x|\leq \frac{1}{2}.
\end{cases}
\]
Then $\g_\delta(x):=(1-\chi(\frac{x}{\delta}))\g(x)$ is a smooth truncation of $\g$ satisfying $\g_\delta(x)=\g(x)$ for $|x|\geq \delta.$ There is thus a unique strong solution to the stochastic differential equation
\begin{equation}\label{eq:truncated-SDE}
\begin{cases}
\diff x_{i,\delta}^t=-\frac{1}{N}\sum\limits_{1\leq  j\leq N;j\neq i}\nabla\g_\delta(x_{i,\delta}^t-x_{j,\delta}^t)\,\diff t+b^t(x_{i,\delta}^t)\,\diff t+\sqrt{2\sigma}\,\diff w_i^t,\\
x_{i,\delta}^t|_{t=0}=x_i^0.
\end{cases}
\end{equation}

Let $\tau_{N,\delta}$ be the stopping time defined by
\[
\tau_{N,\delta}:=\inf\big\{t\geq 0\mid\min_{i\neq j} |x_{i,\delta}^t-x_{j,\delta}^t|\leq 2\delta\big\}.
\]
Then when $0<t\leq \tau_{N,\delta}$ it holds that for all $k\geq 0$ and $1\leq i\neq j\leq N$
\[
 \nabla^k\g_\delta(x_{i,\delta}^t-x_{j,\delta}^t)=\nabla^k\g(x_{i,\delta}^t-x_{j,\delta}^t).
\]
As a consequence, when $\delta'<\delta$, $\tau_{N,\delta}\leq \tau_{N,\delta'}$ and $\ux_{N,\delta}^t=\ux_{N,\delta'}^t$ if $t\leq\tau_{N,\delta}$.
Setting
\[
H_{N,\delta}(\ux_N):=\frac{1}{N^2}\sum_{1\leq i\neq j\leq N} \g_\delta(x_i-x_j),
\]
it is also clear that $H_{N,\delta}(\ux_{N,\delta})=H_N(\ux_{N,\delta})$ when $t<\tau_{N,\delta}.$

We proceed by applying It\^o's formula to $H_{N,\delta}(\ux_{N,\delta})$ to find 
\begin{align}\label{eq:truncIto}
\notag H_{N,\delta}(\ux_{N,\delta}^t)&=H_{N,\delta}(\ux_N^0)-2\int_0^t \frac{1}{N}\sum_{i=1}^N\Bigg|\frac{1}{N}\sum_{1\leq j\leq N:j\neq i}\nabla\g_\delta(x_{i,\delta}^\tau-x_{j,\delta}^\tau)\Bigg|^2\,\diff \tau
\\&\notag\quad+2\int_0^t\frac{1}{N}\sum_{i=1}^N\bigg(\frac{1}{N}\sum_{1\leq j\leq N:j\neq i}\nabla\g_\delta(x_{i,\delta}^\tau-x_{j,\delta}^\tau)\bigg)\cdot b^t(x_{i,\delta}^\tau)\,\diff \tau
\\&\notag\quad +2\sigma \int_0^t\frac{1}{N^2}\sum_{1\leq i\neq j\leq N}\Delta\g_\delta(x_{i,\delta}^\tau-x_{j,\delta}^\tau)\,\diff \tau
\\&\quad+\frac{2\sqrt{2\sigma}}{N}\sum_{i=1}^N\int_0^t\Bigg(\frac{1}{N}\sum_{1\leq j\leq N:j\neq i}\nabla\g_\delta(x_{i,\delta}^\tau-x_{j,\delta}^\tau)\Bigg)\cdot\diff w_i^\tau.
\end{align}
The last term on the right-hand side of \eqref{eq:truncIto} is a martingale with respect to the filtration generated by the noise, which we denote by $M^t$. It has bounded quadratic variation
\[
\langle M\rangle^t=\frac{8\sigma}{N}\int_0^t \frac{1}{N}\sum_{i=1}^N\Bigg|\frac{1}{N}\sum_{1\leq j\neq N:j\neq i}\nabla\g_\delta(x_{i,\delta}^\tau-x_{j,\delta}^\tau)\Bigg|^2\,\diff \tau.
\]

Young's inequality implies that for any $\alpha>0$
\begin{align}\label{eq:drift_Youngs}
\notag&2\bigg|\int_0^t\frac{1}{N}\sum_{i=1}^N\bigg(\frac{1}{N}\sum_{1\leq j\neq N:j\neq i}\nabla\g_\delta(x_{i,\delta}^\tau-x_{j,\delta}^\tau)\bigg)\cdot b(x_{i,\delta}^\tau)\,\diff \tau\bigg|
\\\notag&\quad\leq 2\alpha\int_0^t \frac{1}{N}\sum_{i=1}^N\bigg|\frac{1}{N}\sum_{1\leq j\neq N:j\neq i}\nabla\g_\delta(x_{i,\delta}^\tau-x_{j,\delta}^\tau) \bigg|^2\,\diff \tau+\frac{1}{2\alpha}\int_0^t\frac{1}{N}\sum_{i=1}^N|b^\tau(x_{i,\delta}^\tau)|^2\,\diff \tau
\\&\quad\leq 2\alpha\int_0^t \frac{1}{N}\sum_{i=1}^N\bigg|\frac{1}{N}\sum_{1\leq j\neq N:j\neq i}\nabla\g_\delta(x_{i,\delta}^\tau-x_{j,\delta}^\tau) \bigg|^2\,\diff \tau+\frac{1}{2\alpha}\|b\|_{L^2_tC^1_x}^2,
\end{align}
where $L^2_tC^1_x:=L^2([0,T],C^1(\T^d)).$ Setting
\[
Q_N^t(\ux_N):=H_N(\ux_N^t)-2\sigma\int_0^t D_N(\ux_N^\tau)\,\diff \tau,
\]
rearranging~\eqref{eq:truncIto} and using~\eqref{eq:drift_Youngs} we have found that for $0<t\leq\tau_{N,\delta}$
\begin{equation}\label{eq:energy_martingale_equality}Q_N^t(\ux_{N,\delta})-H_N(\ux_N^0)-\frac{1}{2\alpha}\|b\|^2_{L^2_t C^1_x}\leq M^t-\tfrac{N}{4\sigma}(1-\alpha)\langle M\rangle^t.
\end{equation}

For all $\lambda\in\R$, $\exp(\lambda M^t-\frac{\lambda^2}{2}\langle M\rangle^t)$ is a continuous martingale with
\[\E\Big[\exp\Big(\lambda M^0-\frac{\lambda^2}{2}\langle M\rangle^0\Big)\Big]=1.\]
Thus when $\alpha<1$ by setting $\lambda=\frac{N}{2\sigma}(1-\alpha)$ and applying Lemma~\eqref{lemma:exponential-doobs} to~\eqref{eq:energy_martingale_equality} we find that
\begin{equation}\label{align:energydoobs}
\P\bigg(\sup_{t\in[0, T\wedge\tau_{N,\delta}]} Q_N^t(\ux_{N,\delta})\geq L\bigg)\leq \exp\Big(-\frac{N}{2\sigma}(1-\alpha)\big(L-H_N(\ux_0)-\frac{1}{2\alpha}\|b\|^2_{L^2_t C^1_x}\big)\Big)
\end{equation}
for all $L>0$.

As $\g$ and $(-\Delta)\g$ are both bounded below there exists a constant $C>0$ so that if $\tau_{N,\delta}\leq T$, then
\[
 Q_N^{\tau_{N,\delta}}(\ux_{N,\delta})\geq H_N(\ux_{N,\delta}^{\tau_{N,\delta}})-C\sigma T \geq \frac{\min_{|x|\leq 2\delta}\g(x)}{N^2 }-C(1+\sigma T).
\]
The second inequality above follows since if $\tau_{N,\delta}\leq T$ then there must exist a pair of indices $k\neq \ell$ so that $|x_{k,\delta}^{\tau_{N,\delta}}-x_{\ell,\delta}^{\tau_{N,\delta}}|\leq 2\delta$.
Setting 
\[f(\delta):= \frac{\min_{|x|\leq 2\delta}\g(x)}{N^2}-C(1+\sigma T),\]
then $f(\delta)\rightarrow\infty$ as $\delta\rightarrow 0$. Using~\eqref{align:energydoobs} with $L=f(\delta)$ we thus find that
\begin{align*}
\P\big(\tau_{N,\delta}\leq T \big)\leq \P\Big( Q_N^{\tau_{N,\delta}}(\ux_{N,\delta})\geq f(\delta)\Big)\leq \exp\Big(-\frac{N}{2\sigma}(1-\alpha)\big(f(\delta)-H_N(\ux_0)-\frac{1}{2\alpha}\|b\|^2_{L^2_t C^1_x}\big)\Big).
\end{align*}
As the right-hand side above converges to 0 as $\delta\rightarrow 0$, we can find a sequence $\delta_k\rightarrow 0$ so that
\[
\sum_{k\geq 1} \P\Big(\tau_{N,\delta_k}\leq T \Big)<\infty.
\]
The Borell-Cantelli lemma with the monotonicity of $\tau_{N,\delta}$ in $\delta$ imply that $\lim_{\delta\rightarrow 0}\tau_{N,\delta}>T$ almost surely. Since $\ux_{N,\delta}^t$ and $\ux_{N,\delta'}^t$ agree on $0<t\leq \tau_{N,\delta}$ when $\delta'<\delta$, this allows us to define a unique strong (and weak) solution to~\eqref{eq:SDE} by $\ux_N^t:=\lim_{\delta\rightarrow 0}\ux_{N,\delta}^t$. Noting that
\[Q_N=\sup_{t\in[0,T]}Q_N^t,\]
\eqref{align:energydoobs} also implies that $Q_N(\ux_N)<\infty$ almost surely after taking $\delta\rightarrow 0.$

To show~\eqref{eq:energy-control}, since $b=0$, we can take $\alpha\rightarrow 0$ and $\delta\rightarrow 0$ in~\eqref{align:energydoobs} to find that
\[\P\big(Q_N(\ux_N)\geq L \big)\leq e^{-\frac{N}{2\sigma}\big(L-H_N(\ux_0)\big)}.\]
It thus suffices to prove that~\eqref{cond:initial-convergence} implies that $\lim_{N\rightarrow\infty}H_N(\ux_N^0)=\mathcal{E}(\gamma).$

First, we note that
\begin{equation}\label{eq:energy_equality}
H_N(\ux_N^0)=F(\ux_N^0,\gamma)+2\int \g(x-y)\,\diff (\mu_N^0-\gamma)(x)\,\diff \gamma(y)+\Ec(\gamma).
\end{equation}
Applying Lemma~\ref{lem:renormalized_Holder}, we have the inequality
\[\bigg|\int \g(x-y)\,\diff (\mu_N^0-\gamma)(x)\,\diff \gamma(y)\bigg| \leq C\Big(F_N(\ux_N^0,\gamma)+C\|\gamma\|_{L^\infty(\T^d)}N^{-\beta}\Big)\|\gamma\|_{\H}+C\|\gamma\|_{L^\infty}N^{-\beta},\]
thus the first two terms on the right-hand side of~\eqref{eq:energy_equality} converge to 0.

\end{proof}

\section{Regularity of \texorpdfstring{$S$}{}, \texorpdfstring{$S_N$}{}, and \texorpdfstring{$I_\gamma$}{}}\label{sec:continuity_of_S}

In this section we show the convergence of $S_N(\ux_N,\phi)$ to $S(\mu,\phi)$ for all $\phi\in C^\infty([0,T]\times\T^d)$ when $\mu_N\rightarrow \mu$ in $\Cs^T$ and $Q_N(\ux_N)\leq L$ for some $L\geq 0$. We will use the commutator estimate, Proposition~\ref{prop:renormalized_commutator_estimate}, to control the difference between the last term of $S_N(\ux_N,\phi)$ and $S(\mu,\phi).$

As a warm-up, we first show that $S(\mu,\phi)$ is continuous on the sublevel sets of $Q$. Instead of Proposition~\ref{prop:renormalized_commutator_estimate} here we instead use the non-renormalized commutator estimate.

\begin{proposition}\label{prop:commutator-estimate} There exists $C>0$ so that for every $\rho,\nu\in\Mc(\T^d)$ with $\Ec(|\rho|),\Ec(|\nu|)<\infty$ and Lipschitz vector field  $\psi$
\begin{align*}
&\bigg|\int_{(\T^d)^2}\K_\psi(x,y)\,\diff \rho(x)\,\diff\nu(y)\bigg|
\\&\notag\quad\leq CA_\psi\Big(\|\rho\|_{\H(\T^d)}\|\nu\|_{\H(\T^d)}+\big|\rho(\T^d)\big|\|\nu\|_{\H(\T^d)}+\big|\nu(\T^d)\big|\|\rho\|_{\H(\T^d)}\Big),
\end{align*}
where $A_\psi:=\|\nabla {\psi}\|_{L^\infty(\T^d)}+\big\||\nabla|^{\frac{d-\s}{2}}{\psi}\big\|_{L^{\frac{2d}{d-2-\s}}(\T^d)}$.
\end{proposition}

This is the estimate~\cite[Proposition 3.1]{nguyen_mean-field_2022} adapted to the torus and is needed as a preliminary step to prove Proposition~\ref{prop:renormalized_commutator_estimate}. We give the proof in Subsection~\ref{subseq:commutator_estimate} of the Appendix.

The convergence of $S_N(\ux_N,\phi)$ to $S(\mu,\phi)$ is then essentially a renormalized version of the continuity of $S(\mu,\phi)$: $S_N(\ux_N,\phi)$ and $Q_N(\ux_N)$ are respectively equal to $S(\mu_N,\phi)$ and $Q(\mu_N)$ except the self-interactions of the Diracs are removed. The continuity of $S$ also allows us to show that the sublevel sets of $I_\gamma$ are compact with respect to the $C([0,T],\Pc(\T^d))$ topology.

\subsection{Continuity of \texorpdfstring{$S$}{}}

First, we show that $Q$ is lower semi-continuous.

\begin{lemma}\label{lem:Q-lsc}
$Q$ is a lower semi-continuous function on $\Cs^T.$
\end{lemma}

\begin{proof}
Since $\g$ and $(-\Delta\g)$ are lower semi-continuous and bounded from below, the Portmanteau theorem implies that if $\mu_k\rightarrow \mu$ in $\Pc(\T^d)$ then
\[\liminf_{k\rightarrow\infty} \Ec(\mu_k)\geq \Ec(\mu)\quad\text{and}\quad\liminf_{k\rightarrow\infty} \Dc(\mu_k)\geq \Dc(\mu).\]
When combined with Fatou's lemma, these inequalities imply that if $\mu_k\rightarrow\mu$ in $\Cs^T$, then for all $t\in[0,T]$
\[\liminf_{k\rightarrow\infty} \bigg(\Ec(\mu_k^t)+2\sigma\int_0^t\Dc(\mu_k^\tau)\,\diff \tau \bigg)\geq \Ec(\mu^t)+2\sigma\int_0^t\Dc(\mu^\tau)\,\diff \tau.\]
Taking supremums over time, this implies the lemma.
\end{proof}

\begin{proposition}\label{prop:S-continuity}
For all $\phi\in C^\infty([0,T]\times \T^d)$,  $S(\mu,\phi)$ is a continuous function on sublevel sets of $Q.$
\end{proposition}
\begin{proof}
We fix $\phi$ throughout. The function
\[\mu\mapsto \langle\mu^T,\phi^T\rangle-\langle\mu^0,\phi^0 \rangle  -\int_0^T \langle\mu^t,\partial_t\phi^t+\sigma\Delta\phi^t \rangle\,\diff t-\sigma\int_0^T\langle\mu^t, |\nabla\phi^t|^2\rangle\,\diff t\]
is immediately continuous with respect to the topology on $\Cs^T$. It thus suffices to show that if $\mu_k$ is a sequence of measure trajectories converging to $\mu$ and there exists $L>0$ so that
\begin{equation}\label{eq:uniform-Q-bound}
\sup_{k\geq 1} Q(\mu_k)\leq L,    
\end{equation}
then
\begin{equation}\label{eq:slope-convergence}
\lim_{k\rightarrow\infty}\int_0^T\int_{(\T^d)^2}\K_{\nabla\phi^t}(x,y)\,\diff (\mu_k^t)^{\otimes 2}(x,y)\,\diff t=\int_0^T\int_{(\T^d)^2}\K_{\nabla\phi^t}(x,y)\,\diff (\mu^t)^{\otimes 2}(x,y)\,\diff t.
\end{equation}
We emphasize that Lemma~\ref{lem:Q-lsc} with~\eqref{eq:uniform-Q-bound} imply that $Q(\mu)\leq L,$ thus the right-hand side of~\eqref{eq:slope-convergence} is well-defined.

Applying Proposition~\ref{prop:commutator-estimate} with $\rho=\mu_k-\mu$ and $\nu=\mu_k+\mu$, for all $t\in[0,T]$
\begin{align*}
\bigg|\int\K_{\nabla\phi^t}(x,y)\,\diff \big((\mu_k^t)^{\otimes 2}-(\mu^t)^{\otimes 2}\big)(x,y)\bigg|&\leq CA_{\nabla \phi^t}\|\mu_k^t-\mu^t\|_{\H}\Big(1+\|\mu_k^t\|_{\H}^2+\|\mu^t\|_{\H}^2\Big)^{\frac{1}{2}}
\\&\leq CA_{\nabla \phi^t}\|\mu_k^t-\mu^t\|_{\H}\Big(1+Q(\mu_k)+Q(\mu)\Big)^{\frac{1}{2}}
\\&\leq CA_{\nabla \phi^t}\|\mu_k^t-\mu^t\|_{\H}(1+2L)^{1/2}.
\end{align*}
Integrating this bound we thus have that there exists $C_{\phi,L}>0$ so that
\begin{equation*}
\bigg|\int_0^T\int\K_{\nabla\phi^t}(x,y)\,\diff \big((\mu_k^t)^{\otimes 2}-(\mu^t)^{\otimes 2}\big)(x,y)\,\diff t\bigg|\leq C_{\phi,L}\int_0^T \|\mu_k^t-\mu^t\|_{\H}\,\diff t.
\end{equation*}
To conclude it thus suffices to show that
\begin{equation}\label{eq:L1-time-convergence}
  \lim_{k\rightarrow\infty}\int_0^T \|\mu_k^t-\mu^t\|_{\H}^2\,\diff t=0,
\end{equation}
which holds by interpolating between the convergence of $\mu_k$ to $\mu$ in $\C^T$ and the bounds on the enstrophy terms in $Q(\mu_k).$

First, we show that we can control the $\H$ norm between any two probability measures by interpolating between the Wasserstein-1 metric and the $\dot{H}^{1+\frac{\s-d}{2}}$ norm. Letting $\g_\delta$ be the truncated Riesz potential defined in Proposition~\ref{prop:S-continuity}, for any probability measures $\rho$ and $\nu$
\begin{equation}\label{eq:interpolation_split}\|\rho-\nu\|_{\H}^2=\c_{d,\s}^{-1}\int \g_\delta(x-y)\,\diff (\rho-\nu)^{\otimes 2}(x,y)+\c_{d,\s}^{-1}\int (\g-\g_\delta)(x-y)\,\diff (\rho-\nu)^{\otimes 2}(x,y).
\end{equation}
Since $\g_\delta$ is smooth, there exists $C_\delta>0$ so that
\begin{equation}\label{eq:interpolation_weak_bound}
\int \g_\delta(x-y)\,\diff (\rho-\nu)^{\otimes 2}(x,y)\leq C_\delta d(\rho,\nu)^2.
\end{equation}
Since
\[|x|^{-\s}\indc_{\s>0}-\log|x|\indc_{\s=0}\leq \delta|x|^{-\s-2}\quad\text{when }|x|\leq \delta,\]
\eqref{eq:periodic-correction} implies that for all sufficiently small $\delta>0$
\[|\g-\g_\delta|(x)\leq \g(x)\indc_{|x|\leq \delta}\leq \delta C\big((-\Delta)\g(x)+C\big).\]
Accordingly, we have the inequalities
\begin{equation}\label{eq:Delta_g_CS}
\int (\g-\g_\delta)\,\diff (\rho-\nu)^{\otimes 2}\leq C \delta \Big(\int(-\Delta)\g\,\diff (\rho+\nu)^{\otimes 2}+C\Big)\leq C\delta\Big(\|\rho\|_{\dot H^{1+\frac{\s-d}{2}}}^2+\|\nu\|_{\dot H^{1+\frac{\s-d}{2}}}^2+1\Big),    
\end{equation}
where in the last inequality we used the Cauchy--Schwarz and Young's inequalities. Combining~\eqref{eq:interpolation_split},~\eqref{eq:interpolation_weak_bound}, and~\eqref{eq:Delta_g_CS}, we have that for all $\delta>0$ sufficiently small there exists $C_\delta>0$ so that
\begin{equation}\label{eq:interpolation_bound}\|\rho-\nu\|_{\H}^2\leq C_\delta d(\rho,\nu)^2+C\delta\Big(\|\rho\|_{\dot H^{1+\frac{\s-d}{2}}}^2+\|\nu\|_{\dot H^{1+\frac{\s-d}{2}}}^2+1\Big).
\end{equation}

We may now apply~\eqref{eq:interpolation_bound} to find that
\begin{align*}
    \int_0^T \|\mu_k^t-\mu^t\|_{\H}^2\,\diff t&\leq C_\delta\int_0^Td(\mu_k^t,\mu^t)^2\,\diff t+C\delta\int_0^T\|\mu_k^t\|_{\Hp}^2+\|\mu^t\|_{\Hp}^2\,\diff t +C\delta T
    \\&\leq C_\delta\int_0^Td(\mu_k^t,\mu^t)^2\,\diff t+C\delta (T+\sigma^{-1}L).
\end{align*}
Taking $k\rightarrow\infty$ and then $\delta\rightarrow 0$ we conclude.
\end{proof}

\subsection{Convergence of \texorpdfstring{$S_N$}{}}

Our goal now is to prove renormalized versions of Lemma~\ref{lem:Q-lsc} and Proposition~\ref{prop:S-continuity}, where $Q(\mu)$, $S(\mu,\phi),$ and $\|\mu_k^t-\mu^t\|_{\H}$ are respectively replaced by $Q_N(\ux_N^t)$, $S(\ux_N,\phi),$ and $F_N(\ux_N^t,\mu^t)$.

We begin with adapting the proof of Lemma~\ref{lem:Q-lsc} to show that $Q_N$ satisfy a $\Gamma$-limit lower bound with respect to $Q$.
\begin{lemma}\label{lem:gamma_limit}
Let $\ux_N\in C([0,T],(\T^d)^N)$ so that $\mu_N\rightarrow\mu$ in $\Cs^T$ as $N\rightarrow\infty$. Then
\[\liminf_{N\rightarrow\infty} Q_N(\ux_N)\geq Q(\mu).\]
\end{lemma}
\begin{proof}
For all $M>0$
\[H_N(\ux_N^t)=\int_{(\T^d)^2\setminus\Delta}\g(x-y)\,\diff (\mu_N^t)^{\otimes 2}(x,y)\geq\int_{(\T^d)^2}\g(x-y)\wedge M\,\diff (\mu_N^t)^{\otimes 2}(x,y)-\frac{M}{N}.\]
Since $\g\wedge M$ is a bounded continuous function and $(\mu_N^t)^{\otimes 2}$ weakly converge to $(\mu^t)^{\otimes 2}$ we have
\[\liminf_{N\rightarrow\infty}H_N(\ux_N^t)\geq \int_{(\T^d)^2} \g(x-y)\wedge M\,\diff (\mu^t)^{\otimes 2}(x,y).\]
Taking $M\rightarrow\infty$, the monotone convergence theorem implies that
\[\liminf_{N\rightarrow\infty}H_N(\ux_N^t)\geq \int_{(\T^d)^2} \g(x-y)\,\diff (\mu^t)^{\otimes 2}(x,y).\]
An identical argument implies that 
\[\liminf_{N\rightarrow\infty}D_N(\ux_N^t)\geq \int_{(\T^d)^2} (-\Delta)\g(x-y)\,\diff(\mu^t)^{\otimes 2}(x,y).\]
When combined with Fatou's lemma we thus find that for all $t\in [0,T]$
\[\liminf_{N\rightarrow\infty}\bigg(H_N(\ux_N^t)+2\sigma\int_0^tD_N(\ux_N^\tau)\,\diff \tau\bigg)\geq \Ec(\mu^t)+2\sigma\int_0^t\Dc(\mu^\tau)\,\diff \tau.\]
This concludes the lemma after taking supremums over time.
\end{proof}

We now prove the main proposition of this section. Besides replacing the commutator estimate, Proposition~\ref{prop:commutator-estimate}, with the renormalized commutator estimate, Proposition~\ref{prop:renormalized_commutator_estimate}, some additional technical issues arise in adapting the proof of Proposition~\ref{prop:S-continuity}. First, since the modulated energy $F_N(\ux_N,\mu)$ is only well-defined when $\mu$ is sufficiently regular (we always take $\mu\in L^\infty),$ when a measure trajectory is not in $L^\infty([0,T],L^\infty(\T^d))$ we must appropriately mollify it in space and take advantage of the fact that $\|\mu\|_{L^\infty}$ is paired with a negative power of $N$ in~\eqref{eq:renormalized_commutator_estimate}. Second, in Proposition~\ref{prop:S-continuity} when interpolating between the Wasserstein-1 metric and the $\dot{H}^{1+\frac{\s-d}{2}}$ norm we used that
\[\int_{(\T^d)^2} (-\Delta)\g(x-y)\,\diff \rho(x)\,\diff\nu(y)\leq C\|\rho\|_{\H}\|\nu\|_{\H}.\]
In the equivalent place, we instead use Lemma~\ref{lem:renormalized_Holder}. As we only prove Lemma~\ref{lem:renormalized_Holder} for sub-Coulomb Riesz potentials but $(-\Delta)\g$ corresponds to a Coulomb or super-Coulomb Riesz potential when $\s\geq d-4$, instead of bounding $|\g-\g_{\delta}|$ by $(-\Delta)\g$ we bound it by $(-\Delta)^{\frac{\alpha}{2}}\g$ for some $\alpha>0$ sufficiently small that $\s+\alpha<d-2.$

\begin{proposition}\label{prop:S_N-continuity}
    Let $\ux_N\in C([0,T],(\T^d)^N)$ so that $\mu_N\rightarrow \mu$ in $\Cs^T$ as $N\rightarrow\infty$ and there exists $L>0$ so that
    \[\sup_{N\geq 1}Q_N(\ux_N)\leq L. \]
    Then $Q(\mu)\leq L$ and $\lim_{N\rightarrow\infty} S_N(\ux_N,\phi)=S(\mu,\phi)$ for all $\phi\in C^\infty([0,T]\times \T^d)$.
\end{proposition}

\begin{proof}
Lemma~\ref{lem:gamma_limit} immediately implies that $Q(\mu)\leq L$ (and hence $S(\mu,\phi)$ is well-defined).

As in Proposition~\ref{prop:S-continuity}, it suffices to show that for any $\phi\in C^\infty([0,T]\times\T^d)$
\begin{equation*}
\lim_{k\rightarrow\infty}\int_0^T\int_{(\T^d)^2\setminus\Delta}\K_{\nabla\phi^t}(x,y)\,\diff (\mu_N^t)^{\otimes 2}(x,y)\,\diff t=\int_0^T\int_{(\T^d)^2}\K_{\nabla\phi^t}(x,y)\,\diff (\mu^t)^{\otimes 2}(x,y)\,\diff t.
\end{equation*}
For some mollifier $\eta$, throughout we let $\eta_N(x):=N^{\frac{\beta}{2}}\eta(N^{-\frac{\beta}{2d}}x)$ where $\beta>0$ is as in Proposition~\ref{prop:renormalized_commutator_estimate}. We then let $\mu_{\eta_N}^t:=\mu^t*\eta_N$. Using the triangle inequality we find that
\begin{align}
\notag&\bigg|\int_{(\T^d)^2\setminus\Delta}\K_{\nabla\phi^t}(x,y)\,\diff \big((\mu_N^t)^{\otimes 2}-(\mu^t)^{\otimes 2}\big)(x,y)\bigg|
\\&\quad\leq \Big|\int_{(\T^d)^2\setminus\Delta}\K_{\nabla\phi^t}(x,y)\,\diff \big((\mu_N^t)^{\otimes 2}-(\mu_{\eta_N}^t)^{\otimes 2}\big)(x,y)\bigg|\label{eq:slope-split-1}
\\&\qquad+\bigg|\int_{(\T^d)^2}\K_{\nabla\phi^t}(x,y)\,\diff \big((\mu_{\eta_N}^t)^{\otimes 2}-(\mu^t)^{\otimes 2}\big)(x,y)\bigg|.\label{eq:slope-split-2}
\end{align}

Applying inequality~\eqref{eq:renormalized_commutator_estimate} in Proposition~\ref{prop:renormalized_commutator_estimate},~\eqref{eq:slope-split-1} is bounded by
\[CA_{\nabla\phi^t}\Big( F_N(\ux_N^t,\mu_{\eta_N}^t)+C\|\mu_{\eta_N}^t\|_{L^\infty}N^{-\beta}\Big)^{\frac{1}{2}}\Big(H_N(\ux_N^t)+\|\mu_{\eta_N}^t\|_{\H}^2+1+C\|\mu_{\eta_N}^t\|_{L^\infty}N^{-\beta}\Big)^{\frac{1}{2}}.\]
Since $\|\mu_{\eta_N}\|_{\H}\leq C\|\mu\|_{\H}$, $\|\mu_{\eta_N}\|_{L^\infty}\leq \|\eta_N\|_{L^\infty}\leq N^{\frac{\beta}{2}}$, and $Q_N(\ux_N),Q(\mu)\leq L$, we thus have that~\eqref{eq:slope-split-1} is bounded by
\begin{equation}\label{eq:slope-bound-1}
CA_{\nabla\phi^t}\Big( F_N(\ux_N^t,\mu_{\eta_N}^t)+CN^{-\frac{\beta}{2}}\Big)^{\frac{1}{2}}\Big(L+1\Big)^{\frac{1}{2}}.    
\end{equation}
Proceeding in a similar manner, Proposition~\ref{prop:commutator-estimate} implies that~\eqref{eq:slope-split-2} is bounded by
\begin{equation}\label{eq:slope-bound-2}
   CA_{\nabla\phi^t}\|\mu^t*\eta_N-\mu^t\|_{\H}\Big(L+1\Big)^{\frac{1}{2}}.
\end{equation}

Integrating over time, the bounds~\eqref{eq:slope-bound-1} and~\eqref{eq:slope-bound-2} imply that there exists $C_{\phi,L}>0$ so that
\begin{align*}
&\bigg|\int_0^T\int_{(\T^d)\setminus\Delta}\K_{\nabla\phi^t}(x,y)\,\diff \big((\mu_N^t)^{\otimes 2}-(\mu^t)^{\otimes 2}\big)(x,y)\,\diff t\bigg|
\\&\quad\leq C_{\phi,L}\int_0^T\Big( F_N(\ux_N^t,\mu_{\eta_N}^t)+CN^{-\frac{\beta}{2}}\Big)^{\frac{1}{2}}\,\diff t+C_{\phi,L}\int_0^T\|\mu^t*\eta_N-\mu^t\|_{\H}\,\diff t.
\end{align*}
Since $\|\mu_{\eta_N}^t-\mu^t\|_{\H}\rightarrow 0$ as $N\rightarrow\infty$ and $\|\mu^t*\eta_N-\mu^t\|_{\H}\leq C\|\mu^t\|_{\H}$ for all $t\in[0,T]$, the dominated convergence theorem implies that 
\[\lim_{N\rightarrow\infty}\int_0^T\|\mu^t*\eta_N-\mu^t\|_{\H}\,\diff t=0.\]
To complete the proof of the proposition it thus suffices to show that
\[\lim_{N\rightarrow\infty}\int_0^T F_N(\ux_N^t,\mu_{\eta_N}^t)\,\diff t=0.\]
We show this by interpolating between the Wasserstein-1 metric and the uniform bounds on the discrete enstrophies $\int_0^T D_N(\ux_N^t)\,\diff t.$

We first prove an analogous inequality to~\eqref{eq:interpolation_bound}. Given $\underline{y}_N:=(y_1,\dotsc,y_N)\in(\T^d)^N$ with associated empirical measure $\nu_N$ and $\nu\in\Pc(\T^d)\cap L^\infty(\T^d)$ we will control $F_N(\underline{y}_N,\nu).$ Again letting $\g_\delta$ be the truncated potential defined in Proposition~\ref{prop:S-continuity}, we have
\begin{equation}\label{eq:modualted-energy-split} F_N(\underline{y}_N,\nu)=\int_{(\T^d)^2}\g_\delta(x-y)\,\diff (\nu_N-\nu)^{\otimes 2}(x,y)+\int_{(\T^d)^2\setminus\Delta}(\g-\g_\delta)(x-y)\,\diff (\nu_N-\nu)^{\otimes 2}(x,y),
\end{equation}
where we have used that $\g_\delta(0)=0.$ It still holds that there exists $C_\delta>0$ so that
\begin{equation}\label{eq:mod_energy_weak-bound}\int_{(\T^d)^2}\g_\delta(x-y)\,\diff (\nu_N-\nu)^{\otimes 2}(x,y)\leq C_\delta d(\nu_N,\nu)^2.
\end{equation}
We now let $0<\alpha<(d-2-\s)\wedge 2$.~\eqref{eq:periodic-correction} now implies that there exists $C>0$ so that for all sufficiently small $\delta>0$
\[|\g-\g_\delta|\leq \delta^{\frac{\alpha}{2}}C\big((-\Delta)^{\frac{\alpha}{2}}\g+C\big).\]
We thus find that
\begin{align*}
&\int_{(\T^d)^2\setminus\Delta}(\g-\g_\delta)(x-y)\,\diff(\nu_N-\nu)^{\otimes 2}(x,y)
\\&\quad\leq C\delta^{\frac{\alpha}{2}}\bigg(\frac{1}{N^2}\sum_{1\leq i\neq j\leq N}(-\Delta)^{\frac{\alpha}{2}}\g(y_i-y_j)+2\int(-\Delta)^{\frac{\alpha}{2}}\g(x-y)\,\diff \nu_N(x)\,\diff \nu(y)+\|\nu\|_{\dot H^{1+\frac{\s-d}{2}}} ^2+C\bigg).
\end{align*}
Lemma~\ref{lem:renormalized_Holder} and Young's inequality imply that
\[\bigg|\int(-\Delta)^{\frac{\alpha}{2}}\g(x-y)\,\diff \nu_N(x)\,\diff \nu(y)\bigg|\leq C\bigg(\frac{1}{N^2}\sum_{1\leq i\neq j\leq N}(-\Delta)^{\frac{\alpha}{2}}\g(y_i-y_j)+\|\nu\|_{\dot H^{\frac{\alpha+\s-d}{2}}}^2+C\|\nu\|_{L^\infty}N^{-\beta}\bigg).\]
Since $(-\Delta)^{\frac{\alpha}{2}}\g(x)\leq C((-\Delta)\g(x)+C)$ and $\|\nu\|_{\dot H^{\frac{\alpha+\s-d}{2}}}\leq C\|\nu\|_{\dot H^{1+\frac{\s-d}{2}}}$ we thus have that 
\begin{equation}\label{eq:modulated_energy_interpolation}
\int_{(\T^d)^2\setminus\Delta}(\g-\g_\delta)d(\nu_N-\nu)^{\otimes 2}(x,y)\leq C\delta^{\frac{\alpha}{2}}\Big(D_N(\underline{y}_N)+\|\nu\|_{\dot H^{1+\frac{\s-d}{2}}} ^2+C\|\nu\|_{L^\infty}N^{-\beta}+1\Big).
\end{equation}
Combining~\eqref{eq:modualted-energy-split},~\eqref{eq:mod_energy_weak-bound}, and~\eqref{eq:modulated_energy_interpolation}, we have found that
\begin{equation}\label{eq:renormalized_interpolation} F_N(\underline{y}_N,\nu)\leq C_\delta d(\nu_N,\nu)^2+ C\delta^{\frac{\alpha}{2}}\Big(D_N(\underline{y}_N)+\|\nu\|_{\dot H^{1+\frac{\s-d}{2}}}^2 +C\|\nu\|_{L^\infty}N^{-\beta}+1\Big).
\end{equation}

Applying~\eqref{eq:renormalized_interpolation} and the bounds on $\|\mu_{\eta_N}^t\|_{H^{1+\frac{\s-d}{2}}}$ and $\|\mu_{\eta^N}^t\|_{L^\infty}$ in total we have
\begin{align*}
    \int_0^T F_N(\ux_N^t,\mu_{\eta_N}^t)\,\diff t&\leq C_\delta\int_0^Td(\mu_N^t,\mu_{\eta_N}^t)^2\,\diff t+C\delta^{\frac{\alpha}{2}}\int_0^TD_N(\ux_N^t)+\|\mu^t\|_{\dot H^{1+\frac{\s-d}{2}}}^2\,\diff t +C\delta^{\frac{\alpha}{2}} T
    \\&\leq C_\delta\int_0^Td(\mu_N^t,\mu_{\eta_N}^t)^2\,\diff t+C\delta^{\frac{\alpha}{2}} (\sigma^{-1}L+T).
\end{align*}
Since $\mu_N\rightarrow \mu$ and $\mu_{\eta_N}\rightarrow \mu$ in $\Cs^T$ as $N\rightarrow \infty$,
\[\lim_{N\rightarrow \infty}\sup_{t\in[0,T]}d(\mu_N^t,\mu_{\eta_N}^t)=0,\]
and we conclude the proposition by taking $N\rightarrow\infty$ and then $\delta\rightarrow 0.$
\end{proof}

\subsection{Regularity of the rate function}

In this subsection we prove that the sublevel sets of $I_\gamma$ are compact, thus $I_\gamma$ is a good rate function. First, we need the representation mentioned in Remark~\ref{rem:perturbed_representation}. As this follows~\cite[Lemma 4.8]{dawsont_large_1987} exactly we do not give the full proof. 

\begin{lemma}\label{lem:variation} If $I_\gamma(\mu)<\infty$ then there exists $b\in L^2([0,T],L^2(\mu^t))$ so that $\mu$ is a weak solution to~\eqref{eq:mve+b} where $\mu^t\nabla\g*\mu^t$ is defined by~\eqref{def:mu_grad_mu} and 
\[\sup_{\phi\in C^\infty([0,T]\times\T^d)}S(\mu,\phi)=\frac{1}{4\sigma}\int_0^T\int_{\T^d}|b^t(x)|^2\,\diff \mu^t(x)\,\diff t.\]
\end{lemma}
\begin{proof}[Proof sketch]
The main point is that since
\begin{align*}
S(\mu,\phi)=\langle\partial_t\mu-\sigma\Delta\mu-\nabla\cdot(\mu\nabla\g*\mu),\phi \rangle-\int_0^T\langle\mu^t,|\nabla\phi^t|^2\rangle\,\diff t,
\end{align*}
where the first term is linear in $\phi$ and the second term in quadratic in $\phi$ it holds that
\[\sup_{\phi}S(\mu,\phi)=\sup_{\phi}\frac{|\langle\partial_t\mu-\sigma\Delta\mu-\nabla\cdot(\mu\nabla\g*\mu),\phi \rangle|^2}{4\sigma\int_0^T\langle\mu^t,|\nabla\phi^t|^2\rangle\,\diff t}.\]
The proposition then follows by using the Riesz representation theorem with respect to the Hilbert space defined by the closure of $\{\nabla\phi:\phi\in C^\infty([0,T]\times\T^d)\}$ under the $L^2([0,T],L^2(\mu^t))$ norm.
\end{proof}

It is an immediate consequence of Lemma~\ref{lem:Q-lsc} and Proposition~\ref{prop:S-continuity} that $I_\gamma$ is lower semi-continuous. We thus only need to show that sublevel sets are precompact. We use the following representation of precompact sets of $\Cs^T$ given in~\cite[Lemma 1.3]{gartner_mckean-vlasov_1988}.

\begin{lemma}\label{lem:compact_character} Let $R$ be an arbitrary countable dense subset of $C^0(\T^d).$ Then a subset of $\Cs^T$ is relatively compact if and only if it is contained in a set of the form
\[\bigcap_{\psi\in R} \Big\{\mu\in \Cs^T\ \big|\ \langle \mu,\psi\rangle \in K_\psi\Big\},\]
where $K_\psi$ is a compact subset of $C([0,T],\R)$ for each $\psi\in R.$
\end{lemma}

\begin{proof}[Proof of Item~\ref{item:goodness} in Theorem~\ref{thm:LDP}]
Using Lemma~\ref{lem:compact_character}, it suffices to show that for all $L>0,$ $\psi\in C^\infty(\T^d)$, and $\eps>0$, there exists $\delta>0$ so that
\[\big\{\mu\in\Cs^T\mid I_\gamma(\mu)\leq L\big\}\subset\bigg\{\mu\in\Cs^T\ \Big|\ \sup_{0\leq s\leq t\leq T:|s-t|\leq \delta}|\langle\mu^t-\mu^s,\psi\rangle|\leq \eps\bigg\}.\]
Applying Lemma~\ref{lem:variation} and the triangle inequality we see that
\begin{align*}
    |\langle \mu^t-\mu^s,\psi\rangle |&=\bigg|\int_s^t  \langle \sigma\Delta\mu^\tau-\nabla\cdot(\mu^\tau\nabla\g*\mu^\tau)+\nabla\cdot(b^\tau\mu^\tau),\psi\rangle\,\diff \tau\bigg|
    \\&\leq \bigg|\int_s^t \langle \mu^\tau,\sigma\Delta\psi\rangle\,\diff \tau\bigg|+\bigg| \int_s^t \langle \nabla\cdot(\mu^\tau\nabla\g*\mu^\tau),\psi\rangle \,\diff \tau\bigg|+\bigg|\int_s^t\int \nabla\psi\cdot b^\tau\,\diff \mu^\tau\,\diff \tau\bigg|.
\end{align*}
The first term is bounded by $\|\Delta\psi\|_{L^\infty}(t-s)$. Using Proposition~\ref{prop:commutator-estimate} and that
\[\sup_{t\in[0,T]} \|\mu^t\|_{\H}^2\leq Q(\mu)\leq (2\sigma L+\Ec(\gamma)),\]
we can bound the second term as follows
\[\bigg| \int_s^t \langle \nabla\cdot(\mu^\tau\nabla\g*\mu^\tau),\psi\rangle \,\diff \tau\bigg|\leq C\int_s^t A_{\nabla\psi}\|\mu^\tau\|_{\H}(\|\mu^\tau\|_{\H}+1)\,\diff \tau\leq CA_{\nabla\psi} (t-s)(2\sigma L+\Ec(\gamma)+1).\]
Finally, using Cauchy--Schwarz we have that
\[\bigg|\int_s^t\int \nabla\psi\cdot b^\tau\,\diff \mu^\tau\,\diff \tau\bigg| \leq  \bigg(\int_s^t \int |\nabla\psi|^2\,\diff \mu^\tau\,\diff \tau\bigg)^{1/2}\bigg(\int_s^t\int|b^\tau|^2\,\diff \mu^\tau\,\diff \tau\bigg)^{1/2}\leq (t-s)^{1/2}\|\nabla\psi\|_{L^\infty}(4\sigma L)^{1/2}.\]
Thus if $(t-s)$ is taken to be sufficiently small it can be guaranteed that $|\langle \mu^t-\mu^s,\psi\rangle |\leq \eps.$
\end{proof}

\section{Upper bounds}\label{sec:upper_bounds}

In this section, we prove the LDP upper bounds: Item~\ref{item:LDP-upper-bound} in Theorem~\ref{thm:LDP} and the first inequality in Theorem~\ref{thm:local-LDP}. The proof of Item~\ref{item:LDP-upper-bound} is broken into two parts.  We first show that $\mu_N$ are exponentially tight in $\Cs^T$, and then we show that for all $\mu\in\Cs^T$
\[\lim_{\eps\rightarrow0}\limsup_{N\rightarrow\infty}\frac{1}{N}\log\P\big(\mu_N\in B_\eps(\mu)\big)\leq -I_\gamma(\mu).\]
The latter implies that $\mu_N$ satisfy a weak LDP upper bound, that is Item~\ref{item:LDP-upper-bound} holds when the set $F$ is compact. With exponential tightness, this weak upper bound implies a full upper bound. The first inequality in Theorem~\ref{thm:local-LDP} follows almost immediately from the local upper bound estimates as the modulated energy controls weak convergence.

\subsection{Exponential tightness}
The proof that $\mu_N$ are exponentially tight follows the proof of exponential tightness for systems with regular interactions given in~\cite[Lemma 5.6]{dawsont_large_1987} closely. We however have to use Proposition~\ref{prop:energy-estimate} to appropriately control some terms involving $\g$.

\begin{proposition}
If Assumption~\eqref{cond:initial-convergence} holds, then the empirical trajectories $(\mu_N)_{N\geq1}$ associated to~\eqref{eq:SDE} are exponentially tight in $\Cs^T$. That is, for all $L>0$ there exists a compact set $\mathscr{K}_L\subset \Cs^T$ so that
\[\limsup_{N\rightarrow\infty}\frac{1}{N}\log\P(\mu_N\in \mathscr{K}_L^c)\leq -L.\]
\end{proposition}

\begin{proof}
If we can show that for all $R>0$, $\psi\in C^\infty(\T^d)$, and $\alpha>0$ there exists a compact set $K_{\alpha,\psi}\subset C([0,T],\R)$ so that
\begin{equation}\label{eq:local-exponential-bound}
    \P\big(\langle \mu_N,\ \psi \rangle\in K_{\alpha,\psi}^c, Q_N(\ux_N)\leq R\big)\leq e^{-N\alpha},
\end{equation}
then we can conclude the proposition. Indeed,~\eqref{eq:energy-control} implies that there exists $R_L>0$ so that
\[\limsup_{N\rightarrow\infty}\frac{1}{N}\log\P\big(Q_N(\ux_N)> R_L\big)\leq -L.\]
Letting $\{\psi_\ell\}_{\ell\geq 1}\subset C^\infty(\T^d)$ be a dense subset of $C^0(\T^d)$ and $K_{L+\ell,\psi_\ell}$ the compact set so that~\eqref{eq:local-exponential-bound} holds, then
\[\mathscr{K}_L:=\bigcap_{\ell\geq 1} \Big\{\nu\in \Cs^T\mid\langle \nu,\psi_\ell\rangle\in K_{L+\ell,\psi_\ell}\Big\}\]
is relatively compact by Lemma~\ref{lem:compact_character} and
\begin{align*}
    \P(\mu_N\in K_L^c)&\leq \sum_{\ell\geq 1} \P\big(\langle \mu_N,\psi_\ell\rangle\in K_{L+\ell,\psi_\ell}^c,Q_N(\ux_N\big)\leq R_L\big)+\P\big(Q_N(\ux_N)> R_L\big)
    \\&\leq \sum_{\ell\geq 1} e^{-N(L+\ell)}+\P\big(Q_N(\ux_N)> R_L\big)
    \\&\leq Ce^{-NL} +\P\big(Q_N(\ux_N)>R_L\big).
\end{align*}
We thus find that the proposition holds by our choice of $R_L.$

Proceeding accordingly, we fix $R,\psi$ and $\alpha.$
Applying It\^o's formula to $\langle \mu_N^t,\psi^t\rangle$ we have that
\begin{align*}\langle \mu_N^t,\psi\rangle-\langle \mu_N^s,\psi\rangle&= -\int_s^t\frac{1}{N}\sum_{i=1}^N\nabla\psi(x_i^\tau)\cdot\bigg(\frac{1}{N}\sum_{1\leq j\leq N:j\neq i}^N\nabla\g(x_i^\tau-x_j^\tau)\bigg)\,\diff \tau+\sigma \int_s^t\langle\mu_N^\tau, \Delta\psi\rangle\,\diff \tau 
\\&+\frac{\sqrt{2\sigma}}{N}\sum_{i=1}^N\int_s^t \nabla\psi(x_i^\tau)\cdot \diff w_i^\tau.
\end{align*}
Fixing $s$, the last term above is a martingale with respect to the filtration generated by the noise, which we denote by $M_s^t$. It has bounded quadratic variation
\[\langle M_s \rangle^t=\frac{2\sigma}{N}\int_s^t\langle\mu_N^\tau,|\nabla\psi|^2\rangle\,\diff \tau.\]
Since there exists some constant $C>0$ so that $|x||\nabla\g(x)|\leq C(\g(x)+C)$
\begin{align*}
\frac{1}{N}\sum_{i=1}^N\nabla\psi(x_i)\cdot\bigg(\frac{1}{N}\sum_{1\leq j\leq N:j\neq i}^N\nabla\g(x_i-x_j)\bigg)&=\frac{1}{2N^2}\sum_{1\leq i\neq j\leq N} \big(\nabla\psi(x_i)-\nabla\psi(x_j)\big)\cdot\nabla\g(x_i-x_j)\\
&\leq C \|\nabla\psi\|_{L^\infty}\Big(H_N(\ux_N)+C\Big),
\end{align*}
for all $\ux_N\in (\T^d)^N$. Thus when $Q_N(\ux_N)\leq R$, there exists some $\kappa(d,\s,\sigma,R,\|\nabla\psi\|_{L^\infty},\|\Delta\psi\|_{L^\infty})>0$ so that for all $\alpha>0$
\[\langle \mu_N^t,\psi\rangle-\langle \mu_N^s,\psi\rangle\leq \kappa(1+\alpha)(t-s)+M_s^t-\frac{N\alpha}{2}\langle M_s\rangle^t.\]
The rest of the proof proceeds identically to \cite[Lemma 5.6]{dawsont_large_1987} starting at page 300.
\end{proof}

\subsection{Local upper bound}

Before proving the local upper bounds we precisely state how the modulated energy controls weak convergence. This shows that Assumption~\ref{cond:initial-convergence} guarantees that the initial empirical measures $\mu_N^0$ weakly converge to $\gamma$.

\begin{lemma}\label{lem:weak_control} There exists $C,\beta>0$ so that for any $\psi\in C^\infty(\T^d)$, $\ux_N\in(\T^d)^N$ pairwise distinct, and $\mu\in \Pc(\T^d)\cap L^\infty(\T^d)$ 
\[\bigg|\int_{\T^d} \psi(x)\,\diff \Big(\frac{1}{N}\sum_{i=1}^N\delta_{x_i}-\mu\Big)(x) \bigg|\leq C\Big(\|\nabla \psi\|_{L^\infty(\T^d)}+\|\psi\|_{\dot{H}^\frac{d-\s}{2}(\T^d)}\Big)\Big(F_N(\ux_N,\mu)+C\|\mu\|_{L^\infty(\T^d)}N^{-\beta}\Big)^{1/2}.\]
\end{lemma}
We give the proof in Subsection~\ref{subsec:modulated_energy inequalities} in the Appendix.

\begin{proposition}\label{prop:upper-bound} Given Assumption~\ref{cond:initial-convergence}, for all $\mu\in \Cs^T$
\[\lim_{\eps\rightarrow0}\limsup_{N\rightarrow\infty}\frac{1}{N}\log\P\big(\mu_N\in B_\eps(\mu)\big)\leq -I_\gamma(\mu).\]
\end{proposition}

\begin{proof}
Assumption~\ref{cond:initial-convergence} and Lemma~\ref{lem:weak_control} imply that $\mu_N^0\rightarrow \gamma$ in $\Pc(\T^d)$. Thus if $\mu^0\neq \gamma$, then
\[\limsup_{N\rightarrow\infty}\frac{1}{N}\log\P\big(\mu_N\in B_\eps(\mu)\big)=-\infty\]
for any sufficiently small $\eps>0.$

On the other hand, Lemma~\ref{lem:gamma_limit} implies that for all $\delta>0$ there exists $\eps>0$ so that for all  sufficiently large $N$
\[\Big\{\ux_N\ \big|\ \mu_N\in B_\eps(\mu)\Big\}\subset \Big\{\ux_N\ \big|\ Q_N(\ux_N)>Q(\mu)\wedge\tfrac{1}{\delta} -\delta\Big\}.\]
The inequality~\eqref{eq:energy-control} thus implies that
\[\limsup_{N\rightarrow\infty}\frac{1}{N}\log\P\big(\mu_N\in B_\eps(\mu)\big)\leq -\frac{1}{2\sigma}\big(Q(\mu)\wedge\tfrac{1}{\delta} -\delta-\Ec(\gamma)\big).\]
Taking $\delta\rightarrow 0$ we then find that
\[\lim_{\eps\rightarrow 0}\limsup_{N\rightarrow\infty}\frac{1}{N}\log\P\big(\mu_N\in B_\eps(\mu)\big)\leq -\frac{1}{2\sigma}\big(Q(\mu)-\Ec(\gamma)\big).\]

To complete the proposition it suffices to show that if $Q(\mu)<\infty$, then
\[\lim_{\eps\rightarrow 0}\limsup_{N\rightarrow\infty}\frac{1}{N}\log\P\big(\mu_N\in B_\eps(\mu)\big)\leq -\sup_{\phi\in C^\infty}S(\mu,\phi).\]
Fixing $\phi\in C^\infty$, It\^o's formula gives us the equality
\begin{align*}
\langle \mu_N^t,\phi^t\rangle&=\langle \mu_N^0,\phi^0\rangle -\int_0^t\frac{1}{N}\sum_{i=1}^N \nabla\phi^\tau(x_i^\tau)\cdot\bigg(\frac{1}{N}\sum_{1\leq j\leq N:j\neq i}\nabla\g(x_i^\tau-x_j^\tau)\bigg)\,\diff \tau
\\&\quad+\sigma\int_0^t \langle \mu_N^\tau,\partial_t\phi^\tau+\sigma\Delta\phi^\tau\rangle\,\diff \tau+\frac{\sqrt{2\sigma}}{N}\sum_{i=1}^N\int_0^t \nabla\phi^\tau(x_i^\tau)\cdot \diff w_i^\tau.
\end{align*}
The last term, which we call $M^t$, is a martingale with respect to the filtration generated by the noise with bounded quadratic variation equal to
\[\langle M\rangle^t=\frac{2\sigma}{N}\int_0^t\langle\mu_N^\tau,|\nabla\phi^\tau|^2\rangle\,\diff \tau.\]
Noting that
\[\frac{1}{N}\sum_{i=1}^N \nabla\phi^\tau(x_i^\tau)\cdot\bigg(\frac{1}{N}\sum_{1\leq j\leq N:j\neq i}\nabla\g(x_i^\tau-x_j^\tau)\bigg)=\frac{1}{2}\int_{(\T^d)^2\setminus\Delta}(\nabla\phi^\tau(x)-\nabla\phi^\tau(y))\cdot\nabla\g(x-y)\,\diff (\mu_N^\tau)^{\otimes 2}\]
after rearranging we thus find
\begin{align*}
    S_N(\ux_N,\phi)=M^T-\frac{N}{2}\langle M\rangle^T.
\end{align*}

For any $L>0$ we have the union bound
\[\P\big(\mu_N\in B_\eps(\mu)\big)\leq \P(\mu_N\in B_\eps(\mu),Q_N(\ux_N)\leq L)+\P(Q_N(\ux_N)> L).\]
Thus letting $A_{N,\eps,L}:=\big\{\ux_N\in C([0,T],(\T^d)^N)\mid\mu_N\in B_\eps(\mu),Q_N(\ux_N)\leq L\big\}$, Chebyshev's inequality implies that
\begin{align*}
\P\big(\mu_N\in B_\eps(\mu),Q_N(\ux_N)\leq L\big)&= \E \Big[\exp\Big(-NS_N(\ux_N,\phi)\Big)\exp\Big(N M^t-\tfrac{N^2}{2}\langle M\rangle^t\Big) \indc_{A_{N,\eps,L}}\Big]
\\&\leq \exp\Big(-N\inf_{\ux_N\in A_{N,\eps,L}}S_N(\ux_N,\phi)\Big)\E\Big[\exp(N M^t-\tfrac{N^2}{2}\langle M\rangle^t)\Big]
\\&=\exp\Big(-N\inf_{\ux_N\in A_{N,\eps,L}}S_N(\ux_N,\phi)\Big),
\end{align*}
where the last line follows since $\exp(N M^t-\frac{N^2}{2}\langle M\rangle^t)$ is a martingale with constant expectation equal to 1.
We have thus found that
\begin{align*}&\lim_{\eps\rightarrow0}\limsup_{N\rightarrow\infty}\frac{1}{N}\log\P\big(\mu_N\in B_\eps(\mu)\big)
    \\&\quad\leq \Big( -\lim_{\eps\rightarrow0}\liminf_{N\rightarrow\infty}\inf_{\ux_N\in A_{N,\eps,L}}S_N(\ux_N,\phi)\Big)\vee\Big(\limsup_{N\rightarrow\infty}\frac{1}{N}\log\P(Q_N(\ux_N)>L) \Big).
\end{align*}
Proposition~\ref{prop:S_N-continuity} immediately implies that
\[ \lim_{\eps\rightarrow0}\liminf_{N\rightarrow\infty}\inf_{\ux_N\in A_{N,\eps,L}}S_N(\ux_N,\phi)\geq S(\mu,\phi).\]
In total with~\eqref{eq:energy-control} we have found that
\[\lim_{\eps\rightarrow0}\limsup_{N\rightarrow\infty}\frac{1}{N}\log\P\big(\mu_N\in B_\eps(\mu)\big)\leq -\Big(S(\mu,\phi)\wedge\frac{1}{2\sigma}\big(L-\Ec(\gamma)\big)\Big).\]
Sending $L\rightarrow\infty$ and optimizing over $\phi$ completes the proof.
\end{proof}

\subsection{First inequality in Theorem~\ref{thm:local-LDP}}

\begin{proof}[Proof of first inequality in Theorem~\ref{thm:local-LDP}]

We let $\H_0(\T^d)$ denote the closure of zero-mean $C^\infty(\T^d)$ functions with respect the the $\H(\T^d)$ norm. Then using that the space of zero-mean $L^\infty$ functions compactly embeds into $\H_0(\T^d)$ and the weak continuity of $\mu$, we find that if $t_k\rightarrow t$ as $k\rightarrow \infty$ then
\[\lim_{k\rightarrow\infty}\|\mu^{t_k}-\mu^t\|_{\H}=0.\]
We use this to argue that for all $\eps>0$ there exists $\eps'>0$ so that for sufficiently large $N$
\begin{equation}\label{eq:set_inclusion}
  \bigg\{\sup_{t\in[0,T]} F_N(\ux_N^t,\mu^t)<\eps'\bigg\} \subset \Big\{\mu_N\in B_\eps(\mu)\Big\}.  
\end{equation}
We can then immediately conclude that
\[\lim_{\eps\rightarrow 0}\limsup_{N\rightarrow\infty}\frac{1}{N}\log\P\bigg(\sup_{t\in[0,T]} F_N(\ux_N^t,\mu^t)<\eps\bigg)\leq -\sup_{\phi\in C^\infty} S(\mu,\phi)\]
by Proposition~\ref{prop:upper-bound}.

Suppose~\eqref{eq:set_inclusion} is not true. Then there exists some $\eps>0$, a sequence of particle numbers $N_k$, a sequence of trajectories $\ux_{N_k}$, and a sequence of times $t_k$ so that $N_k\rightarrow\infty$ as $k\rightarrow\infty$, $t_k\rightarrow t$ as $k\rightarrow \infty$, $d(\mu_{N_k}^{t_k},\mu^{t_k})\geq \eps$ for all $k$, and 
\[\lim_{k\rightarrow\infty} F_N(\ux_{N_k}^{t_k},\mu^{t_k})=0.\]
This implies that $d(\mu_{N_k}^{t_k},\mu^t)\rightarrow 0$ as $k\rightarrow\infty$. Indeed, for any $\psi\in C^\infty$ we have that
\[\int \psi\,\diff (\mu_{N_k}^{t_k}-\mu^t)=\int\psi\,\diff (\mu^{t_k}-\mu^t)+\int \psi\,\diff (\mu_{N_k}^{t_k}-\mu^{t_k}). \]
The first term on the right-hand side above goes to zero as $k\rightarrow\infty.$ On the other hand, Lemma~\ref{lem:weak_control} implies that
\[\bigg|\int \psi\,\diff \Big(\frac{1}{N}\sum_{i=1}^N\delta_{x_i}-\mu\Big) \bigg|\leq C\Big(\|\nabla \psi\|_{L^\infty}+\|\psi\|_{\dot{H}^\frac{d-\s}{2}}\Big)\Big(F_N(\ux_{N_k}^{t_k},\mu^{t_k})+C\|\mu\|_{L^\infty}N_k^{-\beta}\Big)^{1/2},\]
thus so does the second term. Accordingly, $\mu_{N_k}^{t_k}$ converges to $\mu^t$ weakly. This creates a contradiction since 
\[d(\mu_{N_k}^{t_k},\mu^t)\geq d(\mu_{N_k}^{t_k},\mu^{t_k})-d(\mu^{t_k},\mu^t)\]
and
\[\liminf_{k\rightarrow\infty}d(\mu_{N_k}^{t_k},\mu^{t_k})-d(\mu^{t_k},\mu^t)\geq \eps.\]
\end{proof}

\section{Mean-field limit}\label{sec:mean_field_limit}\label{sec:LLN} 

In this section, we show that the empirical trajectories associated to the solutions of~\eqref{eq:SDE+b} satisfy a mean-field limit when $b$ and the solution to the McKean--Vlasov equation~\eqref{eq:mve+b} are sufficiently regular. The argument is very similar to as in~\cite{rosenzweig_global--time_2023}, and formally proceeds by applying It\^o's formula to $F_N(\ux_N^t,\mu^t).$ The resulting equality relates the modulated energy to the difference between a martingale and $N$ times the martingale's quadratic variation as well as an integral over time of a commutator term. As Proposition~\ref{prop:commutator-estimate} allows the modulated energy to bound the integrand of this last term, by combining Gronwall's inequality with Lemma~\ref{lemma:exponential-doobs} we achieve exponential bounds on the probability the modulated energy is ever larger than some $\eps>0$.

Here we use that the quadratic variation of the martingale arises naturally in the computation of the It\^o derivative of $F_N(\ux_N^t,\mu^t).$ This term was discarded in~\cite{rosenzweig_global--time_2023}, and its consideration here is crucial for giving probability bounds on the behaviour of the modulated energy over time as opposed to global-in-time bounds on the expectation of the modulated energy.

 As was the case in our \textit{a priori} energy bounds, to make this argument rigorous we need to justify the use of It\^o's formula. To do this we use the same truncated process $\ux_{N,\delta}$ as defined in Proposition~\ref{prop:energy-estimate}, and analogously define the truncated modulated energy 
\[F_{N,\delta}(\ux_N,\mu):=\int_{(\T^d)^2\setminus\Delta}\g_\delta(x-y)\,\diff \bigg(\frac{1}{N}\sum_{i=1}^N\delta_{x_i}-\mu\bigg)^{\otimes 2}(x,y),\]
as well as the truncated kernel
\[\K_{\psi,\delta}(x,y):=(\psi(x)-\psi(y))\cdot\nabla\g_\delta(x-y).\]

The following lemma is the consequence of applying Ito's formula to $F_{N,\delta}(\ux_{N,\delta}^t,\mu^t)$ and rearranging appropriately.

\begin{lemma}\label{lem:modulated-energy-Ito} Let $b\in L^2([0,T],C^1(\T^d))$, $\mu\in \Cs^T\cap L^\infty([0,T],L^\infty(\T^d))$ be a weak solution to \eqref{eq:mve+b}, and $\ux_{N,\delta}$ be the solution to~\eqref{eq:truncated-SDE}. Then
\begin{align*}
& F_{N,\delta}(\ux_{N,\delta}^t,\mu^t)-F_{N,\delta}(\ux_{N,\delta}^0,\mu^0)
\\\notag&\quad=-\frac{2}{N}\sum_{i=1}^N \int_0^t\bigg| \int_{\T^d\setminus \{x_{i,\delta}^\tau\}}\nabla \g_\delta(x_{i,\delta}^\tau-y)\,\diff (\mu_{N,\delta}^\tau-\mu^\tau)(y)\bigg|^2\,\diff \tau
\\&\notag\qquad+\int_0^t \int_{(\T^d)^2\setminus\Delta} \K_{u^\tau_\delta+b^\tau,\delta}(x,y)\,\diff (\mu_{N,\delta}^\tau-\mu^\tau)^{\otimes 2}(x,y)\,\diff \tau
\\&\notag\qquad +2\sigma \int_0^t \int_{(\T^d)^2\setminus \Delta} \Delta \g_\delta(x-y)\,\diff(\mu_{N,\delta}^\tau-\mu^\tau)^{\otimes 2}(x,y)\,\diff \tau
\\&\notag\qquad+\frac{2\sqrt{2\sigma}}{N}\sum_{i=1}^N \int_0^t \int_{\T^d\setminus \{x_{i,\delta}^\tau\}}\nabla \g_\delta(x_{i,\delta}^\tau-y)
\,\diff(\mu_{N,\delta}^\tau-\mu^\tau)(y)\cdot \diff w_i^\tau
\\&\notag\qquad+2\int_0^t \int_{\T^d}\g_\delta*\nabla\cdot\big((u^\tau-u_{\delta}^\tau)\mu^\tau\big)\,\diff (\mu_{N,\delta}^\tau-\mu^{\tau})\,\diff \tau,
\end{align*}
where $\mu_{N,\delta}$ is the empirical trajectory associated to $\ux_{N,\delta}$, $u^t:=-\nabla\g*\mu^t$ and $u_\delta^t:=\nabla\g_\delta*\mu^t.$
\end{lemma}
\begin{proof}
The proof follows that of~\cite[Lemma 6.1 and Lemma  6.2]{rosenzweig_global--time_2023}. The only difference is the additional terms that appear due to the drift $b$. Splitting
\begin{align*}
F_{N,\delta}(\ux_{N,\delta}^t,\mu^t)&=\frac{1}{N^2}\sum_{1\leq i\neq j\leq N} \g_\delta(x_{i,\delta}^t-x_{j,\delta}^t)-\frac{2}{N}\sum_{i=1}^N\g_\delta*\mu^t (x_{i,\delta}^t)+\int \g_\delta(x-y)\,\diff \mu^t(x)\,\diff \mu^t(y)
\\&=:\text{Term}_1 +\text{Term}_2+\text{Term}_3,
\end{align*}
then the drift $b$ contributes the following additional components to each term in the It\^o/differential expansion of $F_{N,\delta}(\ux_{N,\delta}^t,\mu^t)$
\begin{align*}
    &\text{Term}_1: \frac{1}{N^2}\sum_{1\leq i\neq j\leq N}\int_0^t \nabla\g_\delta(x_{i,\delta}^\tau-x_{j,\delta}^\tau)\cdot (b(x_{i,\delta}^\tau)-b(x_{j,\delta}^\tau))\,\diff \tau,
    \\&\text{Term}_2: -\frac{2}{N}\int_0^t\int\nabla\g_\delta(x_{i,\delta}^\tau-y)\cdot b^\tau(y)\,\diff \mu^\tau(y)\,\diff \tau-\frac{2}{N}\sum_{i=1}^N\int_0^t\int\nabla\g(x_{i,\delta}^\tau-y)\cdot b(x_{i,\delta}^\tau)\,\diff \mu^\tau(y)\,\diff \tau,
    \\&\text{Term}_3: 2\int_0^t\int\nabla\g_\delta*\mu^\tau\cdot b^\tau\,\diff \mu^\tau.
\end{align*}
These are readily rearranged into
\[\int_0^t \int_{(\T^d)^2\setminus\Delta} (b^\tau(x)-b^\tau(y))\cdot \nabla \g_\delta(x-y)\,\diff(\mu_{N,\delta}^\tau-\mu^\tau)^{\otimes 2}(x,y)\,\diff \tau,\]
which completes the claim.
\end{proof}

\begin{proposition}\label{prop:LLN}  Suppose that $\mu\in \Cs^T\cap L^\infty([0,T],L^\infty(\T^d))$ is a weak solution to~\eqref{eq:mve+b} with drift $b$ satisfying
 \[
 B_b:=\int_0^T \|b^t\|_{C^1(\T^d)}^2+\big\||\nabla|^{\frac{d-\s}{2}}b^t\big\|_{L^{\frac{2d}{d-2-\s}}(\T^d)}^2\,\diff t<\infty
 \]
 and $\ux_N$ is the solution to~\eqref{eq:SDE+b}. Then there exists $\beta>0$ and $C(T,\sigma,B_b,\|\mu\|_{L^\infty([0,T],L^\infty(\T^d))})>0$ so that for all $\eps>0$
\[\P\bigg(\sup_{t\in[0,T]}F_N(\ux_N^t,\mu^t)>\eps\bigg)\leq \exp\Big(-N\big(C^{-1}\eps-F_N(\ux_N^0,\gamma)-CN^{-\beta}\big)\Big).\]
\end{proposition}

\begin{remark} It is an immediate consequence of Proposition~\ref{prop:LLN} that if $\lim_{N\rightarrow \infty}F_N(\ux_N^0,\gamma)=0$ then
\[\lim_{N\rightarrow\infty} \sup_{t\in[0,T]}F_N(\ux_N^t,\mu^t)=0\]
almost surely. That is to say, a strong pathwise law of large numbers holds for $\mu_N$ with respect to the modulated energy.

\end{remark}

\begin{proof} 
We let
\begin{align*} 
\mathcal{F}_{N,\delta}^t:&=F_{N,\delta}(\ux_{N,\delta}^t,\mu^t)-2\sigma \int_0^t \int_{(\T^d)^2\setminus \Delta} \Delta \g_\delta(x-y)\,\diff(\mu_{N,\delta}^\tau-\mu^\tau)^{\otimes 2}(x,y)\,\diff \tau
\\&\quad-\int_0^t \int_{(\T^d)^2\setminus\Delta} \K_{u^\tau_\delta+b^\tau,\delta}(x,y)\,\diff (\mu_{N,\delta}^\tau-\mu^\tau)^{\otimes 2}(x,y)\,\diff \tau
\\&\quad-2\int_0^t \int_{\T^d}\g_\delta*\nabla\cdot\big((u^\tau-u_{\delta}^\tau)\mu^\tau\big)\,\diff (\mu_{N,\delta}^\tau-\mu^{\tau})\,\diff \tau.
\end{align*}
Then
\[M^t:=\frac{2\sqrt{2\sigma}}{N}\sum_{i=1}^N \int_0^t \int_{\T^d\setminus \{x_{i,\delta}^\tau\}}\nabla \g_\delta(x_{i,\delta}^\tau-y)\,\diff (\mu_{N,\delta}^\tau-\mu^\tau)(y)\cdot \diff w_i^\tau,
\]
is a continuous martingale with respect to the filtration generated by the noise with bounded quadratic variation equal to
\[
\langle M\rangle^t =\frac{8\sigma}{N^2}\sum_{i=1}^N \int_0^t\bigg| \int_{\T^d\setminus \{x_{i,\delta}^\tau\}}\nabla \g_\delta(x_{i,\delta}^\tau-y)\,\diff (\mu_{N,\delta}^\tau-\mu^\tau)(y)\bigg|^2\,\diff \tau,
\]
and Lemma~\ref{lem:modulated-energy-Ito} reads
\[\mathcal{F}_{N,\delta}^t-F_{N,\delta}(\ux_{N,\delta}^0,\gamma)= M^t-\frac{N}{4\sigma}\langle M\rangle^t.\]
Using that
 \[\exp\bigg(\frac{N}{2\sigma} M^t-\frac{1}{2}\Big(\frac{N}{2\sigma}\Big)^2\langle M\rangle^t\bigg)\]
is then a continuous martingale with constant expectation equal to 1, Lemma~\ref{lemma:exponential-doobs} implies that 
\begin{equation}\label{eq:molexpdoobbound}
\P\Big(\sup_{t\in[0,T]}\mathcal{F}_{N,\delta}^t\geq\eps\Big)\leq\exp\Big(-\frac{N}{2\sigma}\big(\eps-F_{N,\delta}(\ux_N^0,\mu^0)\big) \Big).
\end{equation}

We claim that
\begin{equation}\label{eq:initial_modulated_convergence}
\lim_{\delta\rightarrow0}F_{N,\delta}(\ux_N^0,\gamma)=F_N(\ux_N^0,\gamma)
\end{equation}
and
\begin{equation}\label{eq:truncated-path-convergence}
\lim_{\delta\rightarrow0}\sup_{t\in[0,T]}\mathcal{F}_{N,\delta}^t=\sup_{t\in[0,T]}\mathcal{F}_N^t\quad\text{almost surely,}
\end{equation}
where
\begin{align*}
    \mathcal{F}_{N}^t&:=F_N(\ux_N^t,\mu^t)-2\sigma \int_0^t \int_{(\T^d)^2\setminus \Delta} \Delta \g(x-y)\,\diff(\mu_N^\tau-\mu^\tau)^{\otimes 2}(x,y)\,\diff \tau
    \\&\quad-\int_0^t \int_{(\T^d)^2\setminus\Delta} \K_{u^\tau+b^\tau}(x,y)\,\diff (\mu_N^\tau-\mu^\tau)^{\otimes 2}(x,y)\,\diff \tau.
\end{align*}
Once combined with~\eqref{eq:molexpdoobbound}, these claims directly imply that
\begin{equation}\label{eq:exp_doob_bound}
\P\Big(\sup_{t\in[0,T]}\mathcal{F}_N^t\geq\eps\Big)\leq\exp\Big(-\frac{N}{2\sigma}\big(\eps-F_N(\ux_N^0,\mu^0)\big) \Big).
\end{equation}
We first show how the proposition follows from~\eqref{eq:exp_doob_bound}.

We note that
\[\|\nabla u^t\|_{L^\infty}+\big\||\nabla|^{\frac{d-\s}{2}}u^t\big\|_{L^{\frac{2d}{d-\s-2}}}\leq C \|\mu^t\|_{L^\infty},\]
 where we bound the first term using Young's inequality since $\nabla^2\g\in L^1$ and bound the second term using Fourier multipliers. We thus have that
\begin{equation}\label{eq:constant_bound}
\int_0^T A_{u^t+b^t}\,\diff t=\int_0^T\|\nabla(u^t+b^t)\|_{L^\infty}+\big\||\nabla|^{\frac{d-\s}{2}}(u^t+b^t)\big\|_{L^{\frac{2d}{d-2-\s}}}\,\diff t\leq B_b+T\|\mu\|_{L^\infty},    
\end{equation}
where for convenience we set $\|\mu\|_{L^\infty}:= \|\mu\|_{L^\infty([0,T],L^\infty(\T^d))}$. 
The lower bound~\eqref{eq:modulated_energy_positivity} implies that
\begin{equation}\label{eq:positivity_1}
 F_N(\ux_N^t,\mu^t)\geq |F_N(\ux_N^t,\mu^t)|-C\|\mu^t\|_{L^\infty}N^{-\beta}. 
\end{equation}
It also implies that when $\s<d-4$
\begin{equation}\label{eq:positivity_2}-2\sigma \int_0^t \int_{(\T^d)^2\setminus \Delta} \Delta \g(x-y)\,\diff(\mu_N^\tau-\mu^\tau)^{\otimes 2}(x,y)\,\diff \tau\geq -CT\sigma\|\mu\|_{L^\infty} N^{-\beta},
\end{equation}
since $(-\Delta\g)$ is a scalar multiple of the periodic Riesz potential corresponding to parameter $\s+2<d-2$. When $d-4\leq \s<d-2$~\eqref{eq:positivity_2} also holds by~\cite[Proposition 5.6]{de_courcel_sharp_2023}.
Finally,~\eqref{eq:old-commutator-estimate} in Proposition~\ref{prop:renormalized_commutator_estimate} gives the lower bound
\[\int_0^t \int_{(\T^d)^2\setminus\Delta} \K_{u^\tau+b 
^\tau}(x,y)\,\diff (\mu_N^\tau-\mu^\tau)^{\otimes 2}(x,y)\,\diff \tau\geq -C\int_0^t A_{u^\tau+b^\tau}\Big(|F_N(\ux_N^\tau,\mu^\tau)|+\|\mu^\tau\|_{L^\infty}N^{-\beta}\Big)\,\diff \tau.
\]
With~\eqref{eq:constant_bound},~\eqref{eq:positivity_1} and~\eqref{eq:positivity_2} this implies that in total
\begin{equation}\label{eq:modulated-absolute}
    |F_N(\ux_N^t,\mu^t)|-C\int_0^t A_{u^\tau+b^\tau}|F_N(\ux_N^t,\mu^t)|\,\diff t\leq \mathcal{F}_{N}^t+C\|\mu\|_{L^\infty}\big(B_b+1+T(\sigma+\|\mu\|_{L^\infty})\big)N^{-\beta}.
\end{equation}
 Gr\"onwall's inequality implies that if
 \[
 \sup_{t\in[0,T]} |F_N(\ux_N^t,\mu^t)|- C\int_0^t A_{u^\tau+b^\tau} |F_N(\ux_N^\tau,\mu^\tau)|\,\diff \tau\leq \eps e^{-C(B_b+T\|\mu\|_{L^\infty})}\Rightarrow \sup_{t\in[0,T]}  |F_N(\ux_N^t,\mu^t)|\leq \eps.
 \]
The contrapositive of this with~\eqref{eq:modulated-absolute} imply that
\[\sup_{t\in[0,T]}  |F_N(\ux_N^t,\mu^t)|>\eps\Rightarrow \sup_{t\in[0,T]}\mathcal{F}_N^t>\eps e^{-C(B_b+T\|\mu\|_{L^\infty})}-\|\mu\|_{L^\infty}\big(B_b+1+T(\sigma+\|\mu\|_{L^\infty})\big)N^{-\beta}.\]
Combining this with~\eqref{eq:exp_doob_bound} we see the proposition holds.

To conclude, we must thus show that~\eqref{eq:initial_modulated_convergence} and~\eqref{eq:truncated-path-convergence} are true. Following the computations in~\cite[Proposition 6.3]{rosenzweig_global--time_2023}\footnote{Replacing the bound $\|\nabla\g*\mu\|_{L^\infty(\R^d)}\leq C\|\mu\|_{L^\infty(\R^d)}^{\frac{\s+1}{d}}$ by $\|\nabla\g*\mu\|_{L^\infty(\T^d)}\leq C\|\mu\|_{L^\infty(\T^d)}$, assumption (v) by $|x||\nabla\g(x)|\leq C(\g(x)+C),$ and $u$ by $u+b$ for~\eqref{eq:trunc_path_bound_4}.} we see that for any $\ux_N\in(\T^d)^N$, $\mu\in \Pc(\T^d)\cap L^\infty(\T^d)$, and bounded vector field $b$ the following inequalities hold
\begin{equation}\label{eq:trunc_path_bound_1}
|F_{N,\delta}(\ux_N,\mu)-F_N(\ux_N,\mu)\Big|\leq C\|\mu\|_{L^\infty}\delta^{d-\s},
\end{equation}
\begin{equation}\label{eq:trunc_path_bound_2}
\bigg|\int_{(\T^d)^2\setminus \Delta} \Delta (\g_\delta-\g)(x-y)\,\diff (\mu_N-\mu)^{\otimes 2}(x,y)\bigg|\leq C\|\mu\|_{L^\infty}\delta^{d-\s-2},
\end{equation}
\begin{equation}\label{eq:trunc_path_bound_3}
\bigg|\int \g_\delta*\nabla\cdot\big((u-u_{\delta})\mu\big)(x)\,\diff (\mu_N-\mu)(x)\bigg|\leq C\|\mu\|_{L^\infty}^2\delta^{d-1-\s},
\end{equation}
and
\begin{align}\label{eq:trunc_path_bound_4}
&\bigg|\int_{(\T^d)^2\setminus\Delta} \Big(\K_{u_\delta+b,\delta}-\K_{u+b}\Big)(x,y)\,\diff (\mu_N^\tau-\mu^\tau)^{\otimes 2}(x,y)\bigg|
\\\notag&\quad\leq  C\|\mu\|_{L^\infty}\Big((H_N(\ux_N)+C)\delta^{d-2-\s}+(\|\mu\|_{L^\infty}+\|b\|_{L^\infty})\delta^{d-1-\s}\Big),
\end{align}
where $u:=-\nabla\g*\mu$ and $u_{\delta}:=-\nabla\g_\delta*\mu$. The inequality~\eqref{eq:trunc_path_bound_1} directly implies~\eqref{eq:initial_modulated_convergence}. As in Proposition~\ref{prop:energy-estimate}, there exists a family of stopping times $\tau_{N,\delta}$ so that $\tau_{N,\delta}\rightarrow \infty$ as $\delta\rightarrow 0$ almost surely and $\ux_{N,\delta}^t=\ux_N^t$ whenever $t<\tau_{\delta}.$ The bounds~\eqref{eq:trunc_path_bound_1}-\eqref{eq:trunc_path_bound_4} thus imply that when $T<\tau_{N,\delta}$
\[\sup_{t\in[0,T]}|\mathcal{F}_{N,\delta}^t-\mathcal{F}_N^t|\leq C\|\mu\|_{L^\infty}\Big(\delta^{d-\s}+T(\|\mu\|_{L^\infty}+\|b\|_{L^\infty})\delta^{d-1-\s}+(\sigma+\sup_{t\in[0,T]}H_N(\ux_N^t)+1)T\delta^{d-2-\s}\Big).\]
As $\sup_{t\in[0,T]}H_N(\ux_N)<\infty$ almost surely by Proposition~\ref{prop:energy-estimate}, we see that~\eqref{eq:truncated-path-convergence} holds.
\end{proof}

\section{Lower bounds}\label{sec:lower_bounds}

We now use the mean-field limit from Section~\ref{sec:LLN} to prove the LDP lower bound. First, we use a tilting argument to show that if $\mu$ and $b$ satisfy the regularity assumptions of Proposition~\ref{prop:LLN} then a large deviations lower bound holds for the probability that the modulated energy between the the measure trajectory and the particle trajectories is small. We then construct a good family of approximations to recover this lower bound for less regular measure trajectories which we use to show Item~\ref{item:LDP-lower-bound} in Theorem~\ref{thm:LDP}. Finally, we show the second inequality in Theorem~\ref{thm:local-LDP} by showing that the approximating sequences converge in a stronger topology than $\Cs^T.$

\subsection{Local lower bound for regular trajectories}

Here we show that a local lower bound holds when $\mu$ is a sufficiently regular weak solution to~\eqref{eq:mve+b}. The proof is very similar to~\cite[Proposition 2.10]{chen_sample-path_2022}, although here we give estimates on the set where the modulated energy between $\ux_N$ and $\mu$ is small as opposed to just the $C([0,T],\Pc(\T^d))$ distance between $\mu_N$ and $\mu$.

\begin{proposition}\label{prop:lower-bound} Suppose that $\ux_N$ solves \eqref{eq:SDE} with initial conditions satisfying Assumption~\ref{cond:initial-convergence} and $\mu$ is a weak solution to~\eqref{eq:mve+b} that satisfies the conditions of Proposition~\ref{prop:LLN} with $\mu^0=\gamma$. Then
\[\lim_{\eps\rightarrow0}\liminf_{N\rightarrow\infty}\frac{1}{N}\log\P\bigg(\sup_{t\in[0,T]}F_N(\ux_N^t,\mu^t)<\eps\bigg)\geq -\frac{1}{4\sigma}\int_0^T \int_{\T^d} |b^t(x)|^2\,\diff \mu^t(x)\,\diff t.\]
\end{proposition}

\begin{proof} We use the change of measure
\[\frac{\diff\P}{\diff\P_b}:=\exp\Bigg(\frac{1}{\sqrt{2\sigma}}\sum_{i=1}^N\int_0^T b^t(x_i^t)\cdot \diff w_i^t-\frac{1}{4\sigma}\sum_{i=1}^N\int_0^T |b^t(x_i^t)|^2\,\diff t\Bigg).\]
Our conditions on $b$ ensure that $b(x_i^t)$ satisfy the Novikov condition, thus we can use the Girsanov theorem to see that $\ux_N^t$ is a solution to~\eqref{eq:SDE+b} under $\P_b$~\cite[Section 3.5]{karatzas_brownian_1998}.

Letting $A_{N,\eps}:=\{ \sup_{t\in[0,T]}F_N(\ux_N,\mu)<\eps\}$ we have that 
\[
\P\bigg(\sup_{t\in[0,T]}F_N(\ux_N^t,\mu^t)<\eps\bigg)= \E_b\bigg[\indc_{A_{N,\eps}}\frac{\diff\P}{\diff\P_b}\bigg]=\P_b( A_{N,\eps})\E_b\bigg[\frac{\indc_{A_{N,\eps}}}{\P_b(A_{N,\eps})}\frac{\diff\P}{\diff\P_b}\bigg].
\]
Jensen's inequality thus implies that
\begin{equation}\label{eq:jensen}
\frac{1}{N}\log\P\bigg(\sup_{t\in[0,T]}F_N(\ux_N^t,\mu^t)<\eps\bigg)\geq \frac{1}{N}\log\P_b(A_{N,\eps})+\frac{1}{\P_b(A_{N,\eps})}\E_b\bigg[\frac{1}{N}\log \frac{\diff\P}{\diff\P_b} \indc_{A_{N,\eps}}\bigg].
\end{equation}
Proposition~\ref{prop:LLN} implies that $\lim_{N\rightarrow\infty}\P_b(A_{N,\eps})=1$, thus the first term on the right-hand side of \eqref{eq:jensen} converges to 0. To conclude it thus suffices to show that
\begin{equation}\label{eq:intermediate-lower-bound}
    \lim_{\eps\rightarrow0}\liminf_{N\rightarrow\infty}\E_b\bigg[\frac{1}{N}\log \frac{\diff\P}{\diff\P_b} \indc_{A_{N,\eps}}\bigg]\geq -\frac{1}{4\sigma}\int_0^T\int |b^t|^2\,\diff \mu^t\,\diff t.
\end{equation}

Using the definition of $\frac{\diff\P}{\diff\P_b}$ and $A_{N,\eps}$ we have that
\begin{align*}
\E_b\bigg[\frac{1}{N}\log \frac{\diff\P}{\diff\P_b} \indc_{A_{N,\eps}}\bigg]&=\E_b\bigg[\frac{1}{N\sqrt{2\sigma}}\sum_{i=1}^N\int_0^T b^t(x_i^t)\cdot \diff w_i^t\ ; \sup_{t\in[0,T]}F_N(\ux_N^t,\mu^t)<\eps\bigg]
\\&\quad-\E_b\bigg[\frac{1}{4\sigma}\int_0^T \int|b^t(x)|^2\,\diff \mu_N(x) \,\diff t\ ; \sup_{t\in[0,T]}F_N(\ux_N^t,\mu^t)<\eps\bigg].
\end{align*}
H\"older's inequality with the It\^o isometry imply that 
\[\E_b\bigg[\frac{1}{N\sqrt{2\sigma}}\sum_{i=1}^N\int_0^T b^t(x_i^t)\cdot \diff w_i^t\ ; \sup_{t\in[0,T]}F(\ux_N^t,\mu^t)<\eps\bigg]\leq\frac{1}{\sqrt{2N\sigma}}\Bigg(\int_0^T \|b^t\|_{L^\infty}^2 \,\diff t\Bigg)^{1/2},\]
the right-hand side of which goes to 0 as $N\rightarrow\infty$. On the other hand
\begin{align*}
&\Bigg|\E_b\Bigg[\frac{1}{4\sigma}\int_0^T \int|b^t(x)|^2\,\diff \mu_N^t(x) \,\diff t\ ; \sup_{t\in[0,T]}F(\ux_N^t,\mu^t)<\eps\Bigg]-\frac{1}{4\sigma}\int_0^T \int |b^t|^2\,\diff \mu^t\,\diff t\Bigg|
\\&\leq \E_b\Bigg[\frac{1}{4\sigma}\int_0^T\Bigg| \int|b^t(x)|^2\,\diff (\mu_N^t-\mu^t)(x)\Bigg|\,\diff t\ ; \sup_{t\in[0,T]}F(\ux_N^t,\mu^t)<\eps\Bigg]
\\&\qquad +\P_b\bigg(\sup_{t\in[0,T]}F(\ux_N^t,\mu^t)\geq \eps\bigg)\frac{1}{4\sigma}\int_0^T \int_{\T^d} |b^t|^2\,\diff \mu^t\,\diff t.
\end{align*}
The second term on the right-hand side above goes to zero as $N\rightarrow\infty$ again using Proposition~\ref{prop:LLN}. Lemma~\ref{lem:weak_control} implies that
\begin{align*}
&\E_b\Bigg[\frac{1}{4\sigma}\int_0^T\Bigg| \int|b^t(x)|^2d(\mu_N^t-\mu^t)(x)\Bigg|\,\diff t\ ; \sup_{t\in[0,T]}F(\ux_N^t,\mu^t)<\eps\Bigg]
\\&\quad\leq \frac{C}{4\sigma} \int_0^T\big(\|\nabla |b^t|^2\|_{L^\infty}+\||b^t|^2\|_{\dot H^{\frac{d-\s}{2}}}\big)\big(\eps+C\|\mu\|_{L^\infty}N^{-\beta}\big)^{1/2}\,\diff t.
\end{align*}
Clearly $\|\nabla |b^t|^2\|_{L^\infty}\leq \|b^t\|_{C^1}^2$ while the fractional Leibniz rule~\cite[Theorem 7.6.1]{grafakos_modern_2014}\footnote{The estimates are stated for $\R^d$, but they carry over to mean zero functions on $\T^d$ as well.} implies that 
\[
\||b^t|^2\|_{\dot H^{\frac{d-\s}{2}}}\leq \| b^t\|_{C^1(\T^d)} \| b^t\|_{\dot H^{\frac{d-\s}{2}}}.
\]
As $\frac{2d}{d-2-\s}>2$ we can further bound $\| b^t\|_{\dot H^{\frac{d-\s}{2}}}\leq \big\||\nabla|^{\frac{d-\s}{2}}b^t\big\|_{L^{\frac{2d}{d-2-\s}}}$. Using our conditions on $b$ and taking $N\rightarrow\infty$ and then $\eps\rightarrow0$ this proves~\eqref{eq:intermediate-lower-bound}, and thus the claimed lower bound.
\end{proof}

\subsection{Regular approximations}

To show that the rate function is well-behaved relative to our approximating sequence we need to make sense of
\[\int_0^T \int|\nabla\g*\mu^t|^2\,\diff\mu^t\,\diff t.\]
Although this is not necessarily a well-defined integral, when $\mu\in\A$ we can give meaning to it as a product of Sobolev distributions. We use the following proposition to do this.

\begin{proposition}\label{prop:alternate_commutator_estimate} There exists $C>0$ so that for any $f,g,h\in C^\infty(\T^d)$ it holds that
\begin{align}\label{eq:cubic-form-bound}
&\Bigg|\int\nabla\g*f(x)\cdot \nabla\g*g(x) h(x)\,\diff x\Bigg|
\\\notag&\quad\leq C\bigg(\big\||\nabla|^{\frac{1}{2}+\frac{\s-d}{2}}h\big\|_{L^{\frac{6d}{3d-\s-1}}(\T^d)}+
\bigg|\int_{\T^d}h(x)\,\diff x\bigg|\bigg)\big\||\nabla|^{\frac{1}{2}+\frac{\s-d}{2}}f\big\|_{L^{\frac{6d}{3d-\s-1}}(\T^d)}\big\||\nabla|^{\frac{1}{2}+\frac{\s-d}{2}}g\big\|_{L^{\frac{6d}{3d-\s-1}}(\T^d)}.
\end{align}
Consequently, the integral on the left-hand side of~\eqref{eq:cubic-form-bound} extends to a trilinear form on $\{\mu\in\Mc(\T^d):\big\||\nabla|^{\frac{1}{2}+\frac{\s-d}{2}}\mu\big\|_{L^{\frac{6d}{3d-\s-1}}(\T^d)}<\infty\}$ satisfying the bound~\eqref{eq:cubic-form-bound}.
\end{proposition}

\begin{proof} We set $F=\nabla\g*f$ and $G=\nabla\g*g$. We first bound
\[
\bigg|\int F\cdot G h\bigg|\leq \bigg|\int h\bigg|\bigg|\int F\cdot G\bigg|+ \bigg|\int F\cdot G \bigg(h-\int h\bigg)\bigg|
\]
so that $h-\int h$ is zero-mean. To bound the first term we use H\"older's inequality and then Sobolev's inequality to find that
\[\bigg|\int F\cdot G\bigg|\leq C\big\||\nabla|^{1+\s-d}f\big\|_{L^2}\big\||\nabla|^{1+\s-d}g\big\|_{L^2}\leq C\big\||\nabla|^{\frac{1}{2}+\frac{\s-d}{2}}f\big\|_{L^{\frac{6d}{3d-\s-1}}}\big\||\nabla|^{\frac{1}{2}+\frac{\s-d}{2}}g\big\|_{L^{\frac{6d}{3d-\s-1}}}.\]
To bound the second term, H\"older's inequality implies that
\[
\Bigg|\int F\cdot G\bigg(h-\int h\bigg)\Bigg|\leq \big\| |\nabla|^{\frac{d-\s}{2}-\frac{1}{2}}(F\cdot G)\big\|_{L^{\frac{6d}{3d+\s+1}}}\big\||\nabla|^{\frac{1}{2}+\frac{\s-d}{2}}h\big\|_{L^{\frac{6d}{3d-\s-1}}}.
\]
The fractional Leibniz rule in turn gives the bound
\[
\big\||\nabla|^{\frac{d-\s}{2}-\frac{1}{2}}(FG)\big\|_{L^{\frac{6d}{3d+\s+1}}}\leq C\Big(\big\| |\nabla|^{\frac{1}{2}+\frac{\s-d}{2}}f\big\|_{L^{\frac{6d}{3d-\s-1}}}\|G\|_{L^{\frac{3d}{\s+1}}}+\|F\|_{L^{\frac{3d}{\s+1}}}\big\| |\nabla|^{\frac{1}{2}+\frac{\s-d}{2}}g\big\|_{L^{\frac{6d}{3d-\s-1}}}\Big).
\]
Sobolev's inequality implies that
\[\|F\|_{L^{\frac{3d}{\s+1}}}\leq C \big\||\nabla|^{\frac{1}{2}+\frac{\s-d}{2}}f\big\|_{L^{\frac{6d}{3d-\s-1}}},\]
thus using identical bounds on $G$ we have in total found that
\[\Bigg|\int F\cdot G\bigg(h-\int h\bigg)\Bigg|\leq C\big\||\nabla|^{\frac{1}{2}+\frac{\s-d}{2}}h\big\|_{L^{\frac{6d}{3d-\s-1}}}\big\||\nabla|^{\frac{1}{2}+\frac{\s-d}{2}}f\big\|_{L^{\frac{6d}{3d-\s-1}}}\big\||\nabla|^{\frac{1}{2}+\frac{\s-d}{2}}g\big\|_{L^{\frac{6d}{3d-\s-1}}},\]
which completes the claim.
\end{proof}

We now show that we can construct a good approximating sequence for all $\mu\in\A.$ Essentially these are Gaussian mollifications of $\mu$ except we need to modify them near time zero so they have initial conditions equal to $\gamma.$

\begin{proposition}\label{prop:recovery} Suppose $\mu$ is a solution to~\eqref{eq:mve+b} with $I_\gamma(\mu)<\infty$ and $\mu\in\A$. Then setting
\[\nu_\eps^t:=
\begin{cases}
\gamma*\Phi^t&0\leq t\leq \eps,\\
\mu^{t-\eps}*\Phi^\eps &\eps< t\leq T,
\end{cases}  
\]
$\nu_\eps$ satisfies the conditions of Proposition~\ref{prop:LLN}  with drift $b_\eps$ for all $\eps>0$. More so, $\nu_\eps\rightarrow \mu$ in $\Cs^T$ as $\eps\rightarrow0$ and 
\begin{equation}\label{eq:recovery-bound}
    \limsup_{\eps\rightarrow0}\frac{1}{4\sigma}\int_0^T \int_{\T^d} |b^t_\eps(x)|^2\,\diff \nu^t_\eps(x)\,\diff t\leq \sup_{\phi\in C^\infty([0,T]\times\T^d)}S(\mu,\phi).
\end{equation}
\end{proposition}

\begin{proof} Throughout we let $(f)_\eps:=f*\Phi^\eps$ where $\Phi^t$ is the fundamental solution to the heat equation
\[
\begin{cases}
 \partial_t\Phi^t-\sigma\Delta\Phi^t=0,\\
\Phi^0=\delta_0.
\end{cases}
\]

Since $I_\gamma(\mu)<\infty$, Lemma~\ref{lem:variation} implies that there exists $b\in L^2([0,T],L^2(\mu^t))$ so that
\[\frac{1}{4\sigma}\int_0^T \int_{\T^d} |b^t|^2\,\diff \mu^t\,\diff t=\sup_{\phi\in C^\infty}S(\mu,\phi)\]
and $\mu$ is a weak solution to~\eqref{eq:SDE+b}. It is then immediate that $\nu_\eps\in\Cs^T$ and $\nu_\eps$ is a weak solution to
\[\partial_t\nu^t_\eps-\sigma\Delta\nu^t_\eps-\nabla\cdot(\nu^t_\eps\nabla\g*\nu^t_\eps)=-\nabla\cdot(b_\eps\nu^t_\eps),\]
where
\[
b_\eps^t:=
\begin{cases}
\nabla\g*(\gamma)_t&0\leq t\leq\eps\\
\frac{(b^{t-\eps}\mu^{t-\eps})_\eps}{\mu^{t-\eps}_\eps}+\nabla\g*\mu^{t-\eps}_\eps-\frac{(\mu^{t-\eps}\nabla\g*\mu^{t-\eps})_\eps}{\mu^{t-\eps}_\eps}&\eps<t\leq T
\end{cases}
\]
for $\mu_\eps^t:=(\mu^t)_\eps.$

First, we show that $\nu_\eps$ satisfies the conditions of Proposition~\ref{prop:LLN}. It is immediate that $\nu_\eps\in L^\infty([0,T],L^\infty(\T^d)),$ thus we must show that
 \[
 \int_0^T \|\nabla b_\eps^t\|_{L^\infty}^2+\big\||\nabla|^{\frac{d-\s}{2}}b_\eps^t\big\|_{L^{\frac{2d}{d-2-\s}}}^2\,\diff t<\infty.
 \]
We will repeatedly use that $(\mu_\eps)^{-1}\in L^\infty([0,T],C^k(\T^d))$ for any $k$. Indeed, this follows since $\mu^t$ is a probability measure and $\Phi^\eps$ is lower bounded on the torus for all $\eps>0,$ thus $\mu_\eps^t$ is uniformly lower bounded.

We note that $b^t\mu^t\in L^2([0,T],TV(\T^d))$ since
\[\label{eq:b^2 bound}
\int\psi \cdot b^t\,\diff \mu^t\leq \bigg(\int |b^t|^2 \,\diff \mu^t\bigg)^{1/2}\bigg( \int |\psi|^2 \,\diff \mu^t\bigg)^{1/2}\leq \bigg(\int |b^t|^2\,\diff \mu^t\bigg)^{1/2}\|\psi\|_{C^0}.
\]
As a consequence, $(b^t\mu^t)_\eps \in L^2([0,T],C^k(\T^d))$ for any $k$. Combined with the regularity of $(\mu_\eps)^{-1}$ this immediately implies that
\begin{equation}\label{eq:b-mol-bound}
    \int_0^T\bigg\|\frac{(b^t\mu^t)_\eps}{\mu^t_\eps}\bigg\|_{C^1}^2+\bigg\||\nabla|^{\frac{d-\s}{2}}\frac{(b^t\mu^t)_\eps}{\mu^t_\eps}\bigg\|_{L^{\frac{2d}{d-2-\s}}}\,\diff t<\infty. 
\end{equation}

The commutator estimate Proposition~\ref{prop:commutator-estimate} implies that $\mu^t\nabla\g*\mu^t$ as defined by~\eqref{def:mu_grad_mu} is in $L^\infty([0,T], C^k(\T^d)')$ for any $k>\frac{d-\s}{2}$. Using the regularity of $(\mu^t_\eps)^{-1}$ again, we find as a consequence that
\begin{equation}\label{eq:slope-mol-bound}
    \int_0^T\bigg\|\frac{(\mu^t\nabla\g*\mu^t)_\eps}{\mu^t_\eps}\bigg\|_{C^1}^2+\bigg\||\nabla|^{\frac{d-\s}{2}}\frac{(\mu^t\nabla\g*\mu^t)_\eps}{\mu^t_\eps}\bigg\|_{L^{\frac{2d}{d-2-\s}}}\,\diff t<\infty. 
\end{equation}

Since $\nabla\g$ is integrable it is immediate that
\[
    \int_0^T\|\nabla \g*\mu^t_\eps\|_{C^1}^2+\big\||\nabla|^{\frac{d-\s}{2}}\nabla \g*\mu^t_\eps\big\|_{L^{\frac{2d}{d-2-\s}}}^2\,\diff t<\infty.
\]
Together with~\eqref{eq:b-mol-bound} and~\eqref{eq:slope-mol-bound}, this implies that
\[
 \int_\eps^T \|b_\eps^t\|_{C^1}^2+\big\||\nabla|^{\frac{d-\s}{2}}b_\eps^t\big\|_{L^{\frac{2d}{d-2-\s}}}^2\,\diff t<\infty.
 \]

 To conclude that $\nu_\eps$ satisfies the conditions of Proposition~\ref{prop:LLN}, we note that for any $t\geq 0$
\[\|\nabla\g*(\gamma)_t\|_{C^1}^2+\big\||\nabla|^{\frac{d-\s}{2}}\nabla\g*(\gamma)_t\big\|_{L^{\frac{2d}{d-2-\s}}}^2\leq C\|\gamma\|_{L^\infty},\]
thus certainly
\[
 \int_0^\eps \|b_\eps^t\|_{C^1}^2+\big\||\nabla|^{\frac{d-\s}{2}}b_\eps^t\big\|_{L^{\frac{2d}{d-2-\s}}}^2\,\diff t<\infty.
 \]

By construction $\nu_\eps^0=\gamma$. Since for any $\rho\in \Pc(\T^d),$ \[d(\rho*\Phi_\eps,\rho)\leq C\eps\]
the convergence of $\nu_\eps$ to $\mu$  as $\eps\rightarrow 0$ is direct. To conclude the proposition it only remains to show that~\eqref{eq:recovery-bound} holds.

Since
\[\int_0^T \int |b^t_\eps|^2\,\diff \nu^t_\eps\,\diff t=\int_0^\eps \int |\nabla\g*(\gamma)_t|^2\,\diff (\gamma)_t\,\diff t+\int_0^{T-\eps}\bigg|\frac{(b^t\mu^t)_\eps}{\mu_\eps^t}+\nabla\g*\mu^t_\eps-\frac{(\mu^t\nabla\g*\mu^t)_\eps}{\mu_\eps^t}\bigg|^2\,\diff \mu_\eps^t\,\diff t,\]
it suffices to show that
\begin{equation}\label{eq:init_error_vanishes}
\lim_{\eps\rightarrow 0}\int_0^\eps\int |\nabla\g*(\gamma)_t|^2\,\diff (\gamma)_t\,\diff t=0,
\end{equation}
\begin{equation}\label{eq:b-jensen}
\int_0^T\int\bigg|\frac{(b^t\mu^t)_\eps}{\mu_\eps^t}\bigg|^2 \,\diff \mu_\eps^t\,\diff t\leq \int_0^T\int_{\T^d}|b^t|^2d\mu^t\,\diff t,\end{equation}
and
\begin{equation}\label{eq:commutator-vanishes}
\limsup_{\eps\rightarrow0}\int_0^T\int\bigg|\frac{(\mu^t\nabla\g*\mu^t)_\eps}{\mu_\eps^t}-\nabla\g*\mu_\eps^t \bigg|^2\,\diff \mu_\eps^t\,\diff t=0.
\end{equation}
The limit~\eqref{eq:init_error_vanishes} follows  since $\gamma\in L^\infty$ while inequality~\eqref{eq:b-jensen} follows by \cite[Lemma 8.1.10]{ambrosio_gradient_2008}. 

To prove~\eqref{eq:commutator-vanishes} we will use the dominated convergence theorem and that $\mu\in\A$. In particular, we show that the integrand in the time integral converges to 0 pointwise and is dominated by some $L^1([0,T])$ function.

We first show the pointwise convergence. Since $\mu\in\A$, $\big\||\nabla|^{\frac{1}{2}+\frac{\s-d}{2}}\mu^t\big\|_{L^{\frac{6d}{3d-\s-1}}}<\infty$ for almost every $t$, thus we only have to show convergence to $0$ for these times. We begin by expanding out the integral
 \begin{equation}\label{eq:square-expansion}
 \int\bigg|\frac{(\mu^t\nabla\g*\mu^t)_\eps}{\mu_\eps^t}-\nabla\g*\mu_\eps^t \bigg|^2\mu_\eps^t= \int\bigg|\frac{(\mu^t\nabla\g*\mu^t)_\eps}{\mu_\eps^t}\bigg|^2\mu_\eps^t-2\int(\mu^t\nabla\g*\mu^t)_\eps\cdot \nabla\g*\mu^t_\eps+\int |\nabla\g*\mu_\eps^t|^2 \mu_\eps^t.
 \end{equation}
Fixing $\eps>0$, Proposition~\ref{prop:commutator-estimate} implies that $\mu_\delta^t\nabla\g*\mu_\delta^t$ converges to $\mu^t\nabla\g*\mu^t$ in distribution and there exists $C_\eps>0$ so that
\[\sup_{\delta>0}\|(\mu_\delta^t\nabla\g*\mu^t_\delta)_\eps\|_{L^\infty}\leq C_\eps.\]
 It then follows by the dominated convergence theorem that 
$(\mu_\delta^t\nabla\g*\mu_\delta^t)_\eps$ converges to $(\mu^t\nabla\g*\mu^t)_\eps$ in $L^3(\T^d)$. Since $(\mu^t_\delta)_\eps^{-1}$ converges to $(\mu^t)_\eps^{-1}$ in $L^3(\T^d)$ as well, we have that
\[\int\bigg|\frac{(\mu^t\nabla\g*\mu^t)_\eps}{\mu_\eps^t}\bigg|^2\mu_\eps^t=\lim_{\delta\rightarrow0}\int\bigg|\frac{(\mu_\delta^t\nabla\g*\mu_\delta^t)_\eps}{(\mu_\delta)_\eps^t}\bigg|^2(\mu_\delta)_\eps^t.\]
\cite[Lemma 8.1.10]{ambrosio_gradient_2008} and Proposition~\ref{prop:alternate_commutator_estimate} then imply that
\begin{equation}\label{eq:drift-correction-bound}
   \int\bigg|\frac{(\mu_\delta^t\nabla\g*\mu_\delta^t)_\eps}{(\mu_\delta)_\eps^t}\bigg|^2\mu_\eps^t=\lim_{\delta\rightarrow0}\int\bigg|\frac{(\mu_\delta^t\nabla\g*\mu_\delta^t)_\eps}{(\mu_\delta)_\eps^t}\bigg|^2(\mu_\delta)_\eps^t\leq \lim_{\delta\rightarrow\infty}\int|\nabla\g*\mu_\delta^t|^2\mu_{\delta}^t=\int|\nabla\g*\mu^t|^2\,\diff\mu^t. 
\end{equation}
We emphasize that the farthest integral is only defined using Proposition~\ref{prop:alternate_commutator_estimate}.
Proposition~\ref{prop:alternate_commutator_estimate} also implies that
\[\lim_{\eps\rightarrow0}\int (\mu^t\nabla\g*\mu^t)_\eps\cdot \nabla\g*\mu_\eps^t=\int|\nabla\g*\mu^t|^2\,\diff \mu^t\]
and
\[\lim_{\eps\rightarrow0}\int |\nabla\g*\mu_\eps^t|^2\mu_\eps^t=\int|\nabla\g*\mu^t|^2\,\diff \mu^t.\]
Altogether these imply that
\begin{align*}
0&\leq \limsup_{\eps\rightarrow0}\int\bigg|\frac{(\mu^t\nabla\g*\mu^t)_\eps}{\mu_\eps^t}\bigg|^2\mu_\eps^t-2\int(\mu^t\nabla\g*\mu^t)_\eps\cdot \nabla\g*\mu^t_\eps+\int |\nabla\g*\mu_\eps^t|^2 \mu_\eps^t
\\&\leq \int|\nabla\g*\mu^t|^2d\mu^t-2\int|\nabla\g*\mu^t|^2\,\diff\mu^t+\int|\nabla\g*\mu^t|^2\,\diff\mu^t
\\&=0,
\end{align*}
namely
\[\lim_{\eps\rightarrow0}\int\bigg|\frac{(\mu^t\nabla\g*\mu^t)_\eps}{\mu_\eps^t}-\nabla\g*\mu_\eps^t \bigg|^2\mu_\eps^t=0,\]
for almost every $t$.

We now show that~\eqref{eq:square-expansion} is dominated. Proposition~\ref{prop:alternate_commutator_estimate} and~\eqref{eq:drift-correction-bound} imply that
\[\int\bigg|\frac{(\mu_\delta^t\nabla\g*\mu_\delta^t)_\eps}{(\mu_\delta)_\eps^t}\bigg|^2\mu_\eps^t\leq \int|\nabla\g*\mu^t|^2\,\diff\mu^t\leq C\Big(1+\big\||\nabla|^{\frac{1}{2}+\frac{\s-d}{2}}\mu^t\big\|_{L^{\frac{6d}{3d-\s-1}}}^3\Big).\]
Proposition~\ref{prop:alternate_commutator_estimate} also implies that
\begin{align*}
 \bigg|\int (\mu^t\nabla\g*\mu^t)_\eps\cdot \nabla\g*\mu_\eps^t\bigg|&=\bigg|\int \mu^t\nabla\g*\mu^t\cdot \nabla\g*(\mu^t_\eps)_\eps\bigg|
 \\&\leq C(1+\big\||\nabla|^{\frac{1}{2}+\frac{\s-d}{2}}\mu^t\big\|_{L^{\frac{6d}{3d-\s-1}}}^2)\big\||\nabla|^{\frac{1}{2}+\frac{\s-d}{2}}(\mu_\eps^t)_\eps\big\|_{L^{\frac{6d}{3d-\s-1}}}
 \\&\leq C\Big(1+\big\||\nabla|^{\frac{1}{2}+\frac{\s-d}{2}}\mu\big\|_{L^{\frac{6d}{3d-\s-1}}}^3\Big)   
\end{align*}
and
\[\bigg|\int |\nabla\g*\mu_\eps^t|^2\mu_\eps^t\bigg|\leq C\Big(1+\big\||\nabla|^{\frac{1}{2}+\frac{\s-d}{2}}\mu^t_\eps\big\|_{L^{\frac{6d}{3d-\s-1}}}^3\Big)\leq C\Big(1+\big\||\nabla|^{\frac{1}{2}+\frac{\s-d}{2}}\mu^t\big\|_{L^{\frac{6d}{3d-\s-1}}}^3\Big). \]
Together these three inequalities show that
\[\int\bigg|\frac{(\mu^t\nabla\g*\mu^t)_\eps}{\mu_\eps^t}-\nabla\g*\mu_\eps^t \bigg|^2\mu_\eps^t\leq C\Big(1+\big\||\nabla|^{\frac{1}{2}+\frac{\s-d}{2}}\mu^t\big\|_{L^{\frac{6d}{3d-\s-1}}}^3\Big),\]
where we note that the left hand side is in $L^1([0,T])$ since $\mu\in\A$.

As we have shown that~\eqref{eq:square-expansion} converges to 0 for almost every $t$ and is bounded by an integrable function,~\eqref{eq:commutator-vanishes} holds by the dominated convergence theorem.
\end{proof}

\subsection{Lower bound}\label{subsec:lower-bound}

The proof of the lower bound of Theorem~\ref{thm:LDP} follows using Proposition~\ref{prop:lower-bound} and Proposition~\ref{prop:recovery}.

\begin{proof}[Proof of Item~\ref{item:LDP-lower-bound} in Theorem~\ref{thm:LDP}]
We show that if $\mu\in\A$ and $I_\gamma(\mu)<\infty$, then for all $\eps>0$
\[\liminf_{N\rightarrow\infty}\frac{1}{N}\log\P(\mu_N\in B_\eps(\mu))\geq -\sup_{\phi\in C^\infty}S(\mu,\phi).\]
This immediately shows Item~\ref{item:LDP-lower-bound} when $\s>0$. When $\s=0$ the lower bound is completed by verifying that $\A\subset\{I_\gamma<\infty\}$.

Letting $\nu_\delta$ be defined as in Proposition~\ref{prop:recovery}, since $\nu_\delta\rightarrow \mu$ in $\Cs^T$ as $\delta\rightarrow 0$, $B_\frac{\eps}{2}(\nu_\delta)\subset B_\eps(\mu)$ for all sufficiently small $\delta$. Fixing such a $\delta$, as in the proof of the first inequality in Theorem~\ref{thm:local-LDP}, there exists $\eps'>0$ so that for all sufficiently large $N$
\[\bigg\{\sup_{t\in[0,T]}F_N(\ux_N^t,\nu_\delta^t)<\eps'\bigg\}\subset\Big\{\mu_N\in B_{\frac{\eps}{2}}(\nu_\delta)\Big\}.\]
This implies that for all sufficiently large $N$
\[\P\big(\mu_N\in B_\eps(\mu)\big)\geq \P(\mu_N\in B_{\frac{\eps}{2}}(\nu_\delta))\geq \P\bigg(\sup_{t\in[0,T]}F_N(\ux_N^t,\nu_\delta^t)<\eps'\bigg). \]
As $\nu_\delta$ satisfies the conditions of Proposition~\ref{prop:lower-bound}, we then have that
\begin{align*}
  \liminf_{N\rightarrow\infty}\frac{1}{N}\log\P\big(\mu_N\in B_\eps(\mu)\big)&\geq \lim_{\eps'\rightarrow0}\liminf_{N\rightarrow\infty}\frac{1}{N}\log\P\bigg(\sup_{t\in[0,T]}F_N(\ux_N^t,\nu_\delta^t)<\eps'\bigg)
  \\&\geq -\frac{1}{4\sigma}\int_0^T\int |b_\delta^t|^2\,\diff \nu_\delta^t\,\diff t.
\end{align*}
Since
\[\limsup_{\delta\rightarrow0}\frac{1}{4\sigma}\int_0^T\int |b_\delta^t|^2\,\diff \nu_\delta^t\,\diff t\leq \sup_{\phi\in C^\infty}S(\mu,\phi), \]
we immediately find the desired inequality.

To conclude it suffices to show that when $\s=0$, $\A\subset \{\mu\in \Cs^T\mid Q(\mu)<\infty\}$. Indeed, the Gagliardo--Nirenberg--Sobolev inequality implies that
\[\big\||\nabla|^{\frac{1}{2}-\frac{d}{2}}\mu^t\big\|_{L^{\frac{6d}{3d-1}}}\leq \big\||\nabla|^{-\frac{d}{2}}\mu^t\big\|_{L^2}^{1/3}\big\||\nabla|^{1-\frac{d}{2}}\mu^t\big\|_{L^2}^{2/3},\]
thus
\[\int_0^T\big\||\nabla|^{\frac{1}{2}-\frac{d}{2}}\mu^t\big\|_{L^{\frac{6d}{3d-1}}}^3\,\diff t \leq \bigg(\sup_{t\in[0,T]}\|\mu^t\|_{\dot H^{-\frac{d}{2}}}\bigg)\bigg(\int_0^T\|\mu^t\|_{\dot H^{1-\frac{d}{2}}}^2\,\diff t\bigg).\]
The right-hand side is finite whenever $Q(\mu)<\infty$.
\end{proof}
\begin{remark} With Proposition~\ref{prop:upper-bound} the above shows that
\[ -I_\gamma(\mu)\geq\lim_{\eps\rightarrow 0}\limsup_{N\rightarrow\infty}\frac{1}{N}\P(\mu_N\in B_\eps(\mu))\geq\lim_{\eps\rightarrow 0}\liminf_{N\rightarrow\infty}\frac{1}{N}\P(\mu_N\in B_\eps(\mu))\geq -\sup_{\phi\in C^\infty}S(\mu,\phi),\]
for all $\mu\in\A$ with $Q(\mu)<\infty$. Since $\sup_{\phi\in C^\infty}S(\mu,\phi)\leq I_\gamma(\mu)$, this implies Remark~\ref{rem:A_equivalence}.
\end{remark}

\subsection{Second inequality in Theorem~\ref{thm:local-LDP}}

We now complete Theorem~\ref{thm:local-LDP} by proving the second inequality. This again follows by Proposition~\ref{prop:lower-bound} and Proposition~\ref{prop:recovery}, except we must also show that $\nu_\eps$ converges to $\mu$ in the uniform $\H$ topology.

\begin{proof}[Proof of second inequality in Theorem~\ref{thm:local-LDP}]
Since $\mu\in L^\infty([0,T],L^\infty(\T^d))$, it is immediate that $\mu\in\A$ and $Q(\mu)<\infty.$ Without loss of generality we can assume that 
\[\sup_{\phi\in C^\infty([0,T]\times \T^d)}S(\mu,\phi)<\infty.\]

We claim that for all $\eps>0$ there exists $\eps'>0$ so that for all sufficiently small $\delta$ and large $N$
\[\bigg\{\sup_{t\in[0,T]}F_N(\ux_N^t,\nu_\delta^t)<\eps'\bigg\}\subset \bigg\{\sup_{t\in[0,T]}F_N(\ux_N^t,\mu^t)<\eps\bigg\}.\]
We can then conclude the claim since with Propositions~\ref{prop:lower-bound} and~\ref{prop:recovery} this implies that
\begin{align*}
  \liminf_{N\rightarrow\infty}\frac{1}{N}\log\P\Big(\sup_{t\in[0,T]}F_N(\ux_N^t,\mu^t)<\eps\Big)&\geq \liminf_{\delta\rightarrow0}\lim_{\eps'\rightarrow0}\liminf_{N\rightarrow\infty} \frac{1}{N}\log\P\Big(\sup_{t\in[0,T]}F_N(\ux_N^t,\nu_\delta^t)<\eps'\Big)
  \\&\geq -\sup_{\phi\in C^\infty}S(\mu,\phi).
\end{align*}

First, expanding out the definition of $F_N(\ux_N^t,\mu^t)$ and using Lemma~\ref{lem:renormalized_Holder} and Young's inequality we find the inequality
\begin{align*}
    F_N(\ux_N^t,\mu^t)&=F_N(\ux_N^t,\nu_\delta^t)+2 \int \g(x-y)\,\diff (\mu_N^t-\nu_\delta^t)(x)\,\diff (\nu_\delta^t-\mu^t)(y)+\|\nu_\delta^t-\mu^t\|_{\H}^2
    \\&\leq C\big(F_N(\ux_N^t,\nu_\delta^t)+\|\nu_\delta^t-\mu^t\|_{\H}^2+C\|\mu^t\|_{L^\infty}N^{-\beta}\big).
\end{align*}
It thus suffices to show that 
\begin{equation}\label{eq:stronger_recovery}
    \lim_{\delta\rightarrow0}\sup_{t\in[0,T]}\|\nu_\delta^t-\mu^t\|_{\H}=0.
\end{equation}

We recall that in the proof of the first inequality in Theorem~\ref{thm:local-LDP} we showed that if $t_k\rightarrow t$ as $k\rightarrow\infty$ then 
\[\lim_{k\rightarrow\infty}\|\mu^{t_k}-\mu^t\|_{\H}=0,\]
that is $(\mu-1)\in C([0,T],\H_0(\T^d))$. We then note that 
\[\lim_{\delta\rightarrow0}\sup_{t\in[0,T]}\|(\mu)^t_\delta-\mu^t\|_{\H}=0.\]
This follows by the Arzel\`a-Ascoli theorem since $\|(\mu)_{\delta}^t-\mu^t\|_{\H}\rightarrow 0$ for every $t$ while
    \[ \|(\mu^t)_\delta -(\mu^s)_\delta\|_{\H}\leq C\|\mu^t-\mu^s\|_{\H},\]
thus $(\mu)_\delta-1$ is a uniformly continuous family in $C([0,T],\H_0(\T^d)).$ Since $\|(\gamma)_\delta-\gamma\|_{\H}\rightarrow 0$ as $\delta\rightarrow 0,$ this immediately implies~\eqref{eq:stronger_recovery}.
\end{proof}

\appendix

\section{Commutator estimates and modulated energy bounds on the torus}\label{appendix}

Here we prove the claimed inequalities involving the modulated energy. Although we are essentially adapting the proofs in~\cite{nguyen_mean-field_2022} to the torus, we include the proofs for completeness. 

The primary observation that we use is that the modulated energy $F_N(\ux_N,\mu)$ is essentially the $\H$ norm between the empirical measure $\mu_N$ and $\mu$. Of course, this does not actually make sense since $\mu_N$ is not in $\H$ due to the self-interaction of the Diracs with respect to the Riesz potential. However, if we appropriately mollify the empirical measure $\mu_N$ by ``smearing'' the mass of each Dirac onto a less singular set, then this smeared measure lives in $\H$ but is quantitatively close to $\mu_N$. We then use this closeness to recover estimates that hold for measures in $\H$.

Throughout we let $\delta_{x}^{(\eta)}$ denote the uniform probability measure supported on a sphere of radius $\eta$ centered at $x$. We then let $\g^{(\eta)}(x):=\g*\delta^{(\eta)}_0(x)$, and
\[\mu_N^{(\eta)}:=\frac{1}{N}\sum_{i=1}^N\delta_{x_i}^{(\eta)}\]
for some particle configuration $\ux_N\in(\T^d)^N.$

\subsection{Modulated energy inequalities}\label{subsec:modulated_energy inequalities}

First, we restate~\cite[Proposition 2.1]{nguyen_mean-field_2022} in the context of the torus. This shows that the modulated energy is ``monotone under mollification.'' That is, modulo some error terms, $F_N(\ux_N,\mu)$ controls the $\H$ norm between $\mu_N^{(\eta)}$ and $\mu$.

\begin{proposition}\label{prop:modulated_monotonicity}
There exists $\rr_0>0$ and $C>0$ so that for all $0<\eta<\frac{\rr_0}{2}$, $\ux_N\in (\T^d)^N$ pairwise distinct, and $\mu\in\Pc(\T^d)\cap L^\infty(\T^d)$
\begin{align*}
&\frac{1}{N^2}\sum_{\substack{1\leq i\neq j\leq N\\|x_i-x_j|\leq \frac{1}{2}\rr_0}}^N \big(\g(x_i-x_j)-\g^{(\eta)}(x_i-x_j)\big)_++\c_{d,\s}^{-1}\big\|\mu_N^{(\eta)}-\mu\big\|_{\H(\T^d)}^2
\\&\quad\leq F_N(\ux_N,\mu)+C\|\mu\|_{L^\infty(\T^d)}\Big(\eta^2+\frac{\eta^{-\s}-\log\eta}{N}\Big).
\end{align*}
\end{proposition}

\begin{remark} The lower bound~\eqref{eq:modulated_energy_positivity} is an immediate consequence of Proposition~\ref{prop:modulated_monotonicity} with $\eta=N^{-1/d}.$
\end{remark}
\begin{proof} We note that~\eqref{eq:periodic-correction} implies that there exists $C>0$ so that
\[
\ |\nabla^{\otimes k} \g|(x)\leq C\Big( \frac{1}{|x|^{\s+k}}+|\log|x||\indc_{\s=k=0}\Big)\qquad \text{for all } k\geq 0 \text{ and } x\in\T^d\setminus\{0\}.
\]
Since $\s<d-2$, $(-\Delta)\g$ is a constant multiple of the periodic Riesz potential corresponding to parameter $\s+2$, thus~\eqref{eq:periodic-correction} implies that there exists $\rr_0>0$ so that
\begin{equation}\label{eq:subharmonicitiy}
    \Delta\g\leq 0\qquad \text{ in } B_{\rr_0}(0).
\end{equation}
This immediately implies (say by~\cite[Equation (2.1)]{nguyen_mean-field_2022}) that if $|x|+\eta<\rr_0$ then
\begin{equation}\label{eq:mollification_monotonicity}
 \g^{(\eta)}(x)\leq \g(x),   
\end{equation}
and more generally we see that $\g$ satisfies assumptions (1.13) and (1.14) in~\cite{nguyen_mean-field_2022}. The proof of the proposition then follows verbatim as~\cite[Proposition 2.1]{nguyen_mean-field_2022}. We note that since $\mu$ is a probability measure, it must be that $\|\mu\|_{L^\infty}\geq 1,$ which we use to clean up the multiplicative factors in the bound.
\end{proof}

Proposition~\ref{prop:modulated_monotonicity} allows us to immediately show that the modulated energy controls the weak convergence of $\mu_N$ to $\mu$ and that Lemma~\ref{lem:renormalized_Holder} holds.

\begin{proof}[Proof of Lemma~\ref{lem:weak_control}]
We have that for any $\eta>0$
\[\int \psi \,\diff(\mu_N-\mu)=\int \psi \,\diff(\mu_N-\mu_N^{(\eta)})+\int \psi \,\diff(\mu_N^{(\eta)}-\mu)\leq \|\nabla\psi\|_{L^\infty}\eta +\|\psi\|_{\dot H^{\frac{d-\s}{2}}}\|\mu_N^{(\eta)}-\mu\|_{\H}.\]
The proposition is then concluded by bounding the last term using Proposition~\ref{prop:modulated_monotonicity} and taking $\eta=N^{-1/d}.$
\end{proof}

\begin{proof}[Proof of Lemma~\ref{lem:renormalized_Holder}]
We have that
\[\int \g\,\diff (\mu_N-\mu)\,\diff (\nu-\mu)=\int \g\,\diff (\mu_N-\mu_N^{(\eta)})\,\diff (\nu-\mu)+\int \g\,\diff (\mu_N^{(\eta)}-\mu)\,\diff (\nu-\mu).\]
H\"older's inequality and then Proposition~\ref{prop:modulated_monotonicity} imply that
\begin{align*}\bigg|\int \g\,\diff (\mu_N^{(\eta)}-\mu)\,\diff (\nu-\mu)\bigg|&\leq\|\mu_N^{(\eta)}-\mu\|_{\H}\|\nu-\mu\|_{\H}
\\&\leq \Big(F_N(\ux_N,\mu)+C\|\mu\|_{L^\infty}\Big(\eta^2+\frac{\eta^{-\s}-\log\eta}{N}\Big)\Big)^{1/2}\|\nu-\mu\|_{\H}.
\end{align*}
On the other hand
\[\bigg|\int \g\,\diff (\mu_N-\mu_N^{(\eta)})\,\diff (\nu-\mu)\bigg|\leq \|\nabla\g*(\mu-\nu)\|_{L^\infty}\eta\leq \|\mu-\nu\|_{L^\infty}\eta,\]
where we have used Young's convolution inequality and that $\nabla\g$ is in $L^1(\T^d)$. We then conclude by taking $\eta=N^{-1/d}.$
\end{proof}

Finally, the first term on the left-hand side of the inequality in Proposition~\ref{prop:modulated_monotonicity} allows us to control the microscale interactions between particles by the modulated energy.

\begin{corollary}\label{cor:microscale_control} There exists $C>0$ so that for all sufficiently small $\eps>0$, $\ux_N\in(\T^d)^N$ pairwise distinct, and $\mu\in\Pc(\T^d)\cap L^\infty(\T^d)$
\[\frac{1}{N^2}\sum_{\substack{1\leq i\leq j\leq N\\|x_i-x_j|\leq \eps}} \Big(\g(x_i-x_j)\indc_{\s>0}+\indc_{\s=0}\Big)\leq CF_N(\ux_N,\mu)+C\|\mu\|_{L^\infty}\Big(\eps+\frac{\eps^{-\s}-\log\eps}{N}\Big).\]
\end{corollary}

\begin{proof}
One can readily check that~\eqref{eq:periodic-correction} implies that there exists $C>0$ so that if $\eps$ is sufficiently small
\[0\leq  C^{-1}\Big(\g(x)\indc_{\s>0}+\indc_{\s=0}\Big)\leq \g(x)-\g^{(3\eps)}(x)+C\eps.\]
for all $|x|<\eps.$
Proposition~\ref{prop:modulated_monotonicity} then immediately implies the claim with $\eta=3\eps.$
\end{proof}

\subsection{The commutator estimate}\label{subseq:commutator_estimate}

We now show Proposition~\ref{prop:commutator-estimate} which is analogous to~\cite[Proposition 3.1]{nguyen_mean-field_2022} adapted to the periodic setting. As we only consider potentials that are exact solutions to~\eqref{eq:Riesz}, we give a simpler proof that essentially follows by repeated integration by parts.

In the proof we will view
\[\int (\psi(x)-\psi(y))\cdot\nabla\g(x-y)\,\diff\rho(x)\,\diff\nu(y)\]
as the extension of a continuous bilinear functional on smooth functions. However, using a straightforward measure theory argument one can easily verify that if $\mu$ and $\nu$ are positive measures with finite Riesz energy then
\[\lim_{\eps\rightarrow 0}\int_{(\T^d)^2}\K_\psi(x,y)\,\diff \mu_\eps(x)\,\diff \nu_\eps(y)=\int_{(\T^d)^2}\K_\psi(x,y)\,\diff \mu(x)\,\diff \nu(y),\]
where $\mu_\eps:=\mu*\phi_\eps$ and $\nu_\eps:=\nu*\phi_\eps$ for a family of standard mollifiers $\phi_\eps$. This shows that the measure-theoretic and functional analytic definitions are consistent.

\begin{proof}[Proof of Proposition~\ref{prop:commutator-estimate}]
By approximation, we can assume that $\psi$ is smooth. We may also assume that $d\mu(x)=f(x)\,\diff x$ and $d\nu(x)=g(x)\,\diff x$ where $f,g\in C^\infty(\T^d)$.

Since $\nabla\g$ is zero-mean
\begin{align}\label{eq:mean_removal}
  \int\K_{\psi}(x,y)f(x)g(y)\,\diff x\,\diff y&=\int\K_{\psi}(x,y)\bigg(f(x)-\int f(z)\,\diff z\bigg)\bigg(g(x)-\int g(z)\,\diff z\bigg)\,\diff x\,\diff y
  \\&\notag\quad+\int f(z)\,\diff z\int\psi(x)\cdot\nabla\g(x-y)g(y)\,\diff x\,\diff y
  \\&\notag\quad-\int g(z)\,\diff z\int\psi(y)\cdot\nabla\g(x-y)f(x)\,\diff x\,\diff y.
\end{align}
We then bound
\begin{align*}
 \bigg|\int\psi(x)\cdot\nabla\g(x-y)g(y)\,\diff x\,\diff y\bigg|=\bigg|\int g(y)\g*\nabla\cdot\psi(y)\,\diff y\bigg|&\leq \|g\|_{\H}\|\g*\nabla\cdot\psi\|_{\dot H^{\frac{d-\s}{2}}}
 \\&=\|g\|_{\H}\|\nabla\cdot\psi\|_{\H}
 \\&\leq C\|g\|_{\H}\|\nabla\psi\|_{L^\infty}.
\end{align*}
As the last term in~\eqref{eq:mean_removal} can be bounded identically, to conclude the claim it suffices to show that if $f$ and $g$ are zero-mean functions then
\[\bigg| \int\K_{\psi}(x,y)f(x)g(y)\,\diff x\,\diff y\bigg|\leq C\Big(\|\nabla\psi\|_{L^\infty}+\||\nabla|^{\frac{d-\s}{2}}\psi\|_{L^{\frac{2d}{d-2-\s}}}\Big)\|f\|_{\H}\|g\|_{\H}.
\]

After expanding out the definition of $\K_\psi$ and using~\eqref{eq:Riesz} we have that
\[\int\K_{\psi}(x,y)f(x)g(y)\,\diff x\,\diff y=\c_{d,\s}^{-1}\int\psi(x)\cdot\Big(\nabla \g*f(-\Delta)^{\frac{d-\s}{2}}\g*g+\nabla \g*g(-\Delta)^{\frac{d-\s}{2}}\g*f\Big)(x)\,\diff x.\]
Since $\|\g*f\|_{\dot{H}^{\frac{d-\s}{2}}}=C\|f\|_{\H}$ and $\g*f$ is zero-mean and in $C^\infty(\T^d)$ it thus suffices to prove that for any $0<\alpha\leq\frac{d}{2}$ there exists $C(d,\alpha)>0$ so that for all zero-mean $F,G\in C^\infty(\T^d)$
\begin{equation}\label{eq:commutator-equivalence}\bigg|\int\psi\cdot\Big(\nabla F(-\Delta)^{\frac{\alpha}{2}}G+\nabla G(-\Delta)^{\frac{\alpha}{2}}F\Big)\bigg|\leq C\Big(\|\nabla\psi\|_{L^\infty}+\big\||\nabla|^{\alpha}\psi\big\|_{L^{\frac{d}{\alpha-1}}}\indc_{\alpha>1}\Big)\|F\|_{\dot{H}^{\alpha}}\|G\|_{\dot{H}^{\alpha}}.
\end{equation}
Letting $\alpha=m+\beta$ with $m\in\N$ and $0<\beta\leq 1$, we prove this inductively in $m$.

We frequently use the Caffarelli-Silvestre extension on the torus~\cite{roncal_fractional_2016}. For $0<\beta\leq 1$ we let $k=1$ when $\beta<1$ and $k=0$ when $\beta=1$. Letting $z:=(x,y)\in \T^d\times \R^k$ and $\delta_{\T^d\times\{0\}}$ the uniform surface measure of $\T^d\times\{0\}$, then for all zero-mean $H\in C^\infty$ there exists an extension $\Psi(H)$ on $\T^d\times \R^k$ so that
\[\big((-\Delta)^\beta H\big)\delta_{\T^d\times\{0\}}=-\nabla_z\cdot\big(|y|^{1-2\beta}\nabla_z \Psi(H)\big).\]
Accordingly, integrating by parts it holds that
\begin{equation}\label{eq:sobolev-equality}
    \int_{\T^d\times \R^k} |\nabla_z\Psi(H)|^2 |y|^{1-2\beta}=\int_{\T^d} H(-\Delta)^\beta H=\|H\|_{\dot H^{\beta}}^2.
\end{equation}
For convenience, we set $\gamma:=1-2\beta$.
\vspace{4mm}

\noindent \textbf{Base case:} First we consider $m=0$. Abusing notation so that ${\psi}$ denotes both ${\psi}(x)$ and ${\psi}(x,y):=({\psi}(x),0)$, by integrating by parts we find that
\begin{align*}
&\int_{\T^d} {\psi}\cdot\big(\nabla F(-\Delta)^{\beta}G+\nabla G(-\Delta)^{\beta}F\big)
\\&\quad=-\int_{\T^d\times \R^k} {\psi}\cdot\big(\nabla_z \Psi(F)\nabla_z\cdot(|y|^\gamma \nabla_z \Psi(G))+\nabla_z \Psi(G)\nabla_z\cdot(|y|^\gamma \nabla_z \Psi(F))\big)
\\&\quad=\int_{\T^d\times \R^k}\nabla {\psi}:\Big(\nabla_z\Psi(F)\otimes \nabla_z\Psi(G)+\nabla_z\Psi(G)\otimes \nabla_z\Psi(F)-\Id\nabla_z\Psi(F)\cdot \nabla_z\Psi(G) \Big)|y|^\gamma.
\end{align*}
Applying Cauchy--Schwarz we can bound the absolute value of the last line by
\[C\|\nabla {\psi}\|_{L^\infty}\bigg(\int_{\T^d\times \R^k} |\nabla_z\Psi(F)|^2 |y|^{\gamma}\bigg)^{1/2}\bigg(\int_{\T^d\times \R^k} |\nabla_z\Psi(G)|^2 |y|^{\gamma}\bigg)^{1/2}.\]
With the equality~\eqref{eq:sobolev-equality} this is exactly~\eqref{eq:commutator-equivalence}.\vspace{4mm}

\noindent \textbf{Induction step:} Suppose that the inequality \eqref{eq:commutator-equivalence} holds for $m$. Then integrating by parts we find that
\begin{align*}
&\int_{\T^d} {\psi}\cdot \Big(\nabla F (-\Delta)^{m+1+\beta} G+\nabla G (-\Delta)^{m+1+\beta} F\Big)
\\&\quad=\sum_{i=1}^d\int_{\T^d} {\psi}_i \Big(\partial_i F (-\Delta)^{m+1+\beta} G+\partial_i G(-\Delta)^{m+1+\beta} F\Big) 
\\&\quad=\sum_{i=1}^d\int_{\T^d} \nabla({\psi}_i \partial_i F)\cdot\nabla (-\Delta)^{m+\beta} G+\nabla({\psi}_i\partial_i G)\cdot\nabla(-\Delta)^{m+\beta} F
\\&\quad=\sum_{j=1}^d\int_{\T^d} {\psi}\cdot \nabla (\partial_j F)(-\Delta)^{m+\beta} \partial_j G+{\psi}\cdot \nabla(\partial_jG)(-\Delta)^{m+\beta} \partial_jF
\\&\quad\qquad+\sum_{i,j=1}^d\int_{\T^d} \partial_j{\psi}_i \partial_i F(-\Delta)^{m+\beta}\partial_j G+\partial_j{\psi}_i\partial_i G(-\Delta)^{m+\beta}\partial_j F.
\end{align*}
We can bound the first term in the last line above using the inductive hypothesis and then the Sobolev inequality to find that
\begin{align*}
&\bigg|\int_{\T^d} {\psi}\cdot \nabla (\partial_j F)(-\Delta)^{m+\beta} \partial_j G+{\psi}\cdot \nabla(\partial_jG)(-\Delta)^{m+\beta} \partial_jF\bigg|
\\&\quad\leq C\Big(\big\||\nabla|^{m+\beta}{\psi}\big\|_{L^{\frac{d}{(m+\beta)-1}}}\indc_{m>1}+\|\nabla {\psi}\|_{L^\infty}\Big)\|\partial_j F\|_{\dot{H}^{m+\beta}}\|\partial_j G \|_{\dot{H}^{m+\beta}}
\\&\quad\leq C\Big(\big\||\nabla|^{m+1+\beta}{\psi}\big\|_{L^{\frac{d}{(m+1+\beta)-1}}}\indc_{m>1}+\|\nabla {\psi}\|_{L^\infty}\Big)\|F\|_{\dot{H}^{m+1+\beta}}\|G \|_{\dot{H}^{m+1+\beta}}.
\end{align*}

To bound the remaining terms we integrate by parts
\begin{align*}
\int_{\T^d} \partial_j {\psi} \partial_i F\partial_j(-\Delta)^{m+\beta} G&=\int_{\T^d} \nabla^m(\partial_j{\psi}_i\partial_j F):(-\Delta)^\beta\nabla^m\partial_j G
\\&=-\int_{\T^d\times\R^k} \Psi\big(\nabla^m(\partial_j{\psi}_i\partial_j f)\big):\nabla_z\cdot\big(|y|^\gamma \nabla_z\Psi\big(\nabla^m\partial_j G\big)\big)
\\&=\int_{\T^d\times\R^k}\nabla_z\Psi\big(\nabla^m(\partial_j{\psi}_i\partial_j F)\big):\nabla_z\Psi\big(\nabla^m\partial_j G\big) |y|^\gamma.
\end{align*}
Applying Cauchy--Schwarz we find that
\begin{align*}
    &\bigg|\int_{\T^d\times\R^k}\nabla_z\Psi\big(\nabla^m(\partial_j{\psi}_i\partial_j F)\big):\nabla_z\Psi\big(\nabla^m\partial_j G\big) |y|^\gamma\bigg|
    \\&\quad\leq \bigg(\int_{\T^d\times\R^k}|\nabla_z\Psi\big(\nabla^m(\partial_j{\psi}_i\partial_j F)\big)|^2|y|^\gamma\bigg)^{1/2}\bigg(\int_{\T^d\times\R^k}|\nabla_z\Psi\big(\nabla^m\partial_j G\big)|^2|y|^\gamma\bigg)^{1/2}
    \\&\quad=\|\nabla^m(\partial_j{\psi}_i\partial_j F)\|_{\dot H^\beta}\|\nabla^m\partial_j G\|_{\dot H^\beta}
    \\&\quad\leq C \|\nabla^m(\partial_j{\psi}_i\partial_j F)\|_{\dot H^\beta}\|G\|_{\dot H^{m+1+\beta}}.
\end{align*}
The fractional Leibniz rule then implies that
\[\|\nabla^m(\partial_j{\psi}_i\partial_j F)\|_{\dot H^\beta}\leq C\Big( \big\||\nabla|^{m+1+\beta} {\psi}\big\|_{L^{\frac{d}{(m+1+\beta)-1}}}\|\partial_j F\|_{L^{\frac{2d}{d-2(m+\beta)}}}+\|\nabla {\psi}\|_{L^\infty}\|F\|_{\dot{H}^{m+1+\beta}}\Big),\]
and Sobolev's inequality gives that
\[\|\partial_j F\|_{L^{\frac{2d}{d-2(m+\beta)}}}\leq C\|F\|_{\dot H^{m+1+\beta}}.\]
Together these imply that
\[\bigg|\int_{\T^d} \partial_j \psi_i\partial_i F(-\Delta)^{m+\beta} \partial_jG\bigg|\leq C\Big(\big\||\nabla|^{m+1+\beta} {\psi}\big\|_{L^{\frac{d}{(m+1+\beta)-1}}}+\|\nabla {\psi}\|_{L^\infty}\Big)\|F\|_{\dot{H}^{m+1+\beta}}\|G\|_{\dot H^{m+1+\beta}}.\]
A symmetric argument gives the bound
\[\bigg|\int_{\T^d}\partial_j{\psi}_i\partial_i G(-\Delta)^{m+\beta}\partial_j F\bigg|\leq C\Big(\big\||\nabla|^{m+1+\beta} {\psi}\big\|_{L^{\frac{d}{(m+1+\beta)-1}}}+\|\nabla {\psi}\|_{L^\infty}\Big)\|F\|_{\dot{H}^{m+1+\beta}}\|G\|_{\dot H^{m+1+\beta}},\]
and we have completed the induction.
\end{proof}

\subsection{Renormalized commutator estimates}\label{subseq:renormalized_commutator_estimate}

We conclude the Appendix by using Proposition~\ref{prop:modulated_monotonicity}, Proposition~\ref{prop:commutator-estimate}, and Corollary~\ref{cor:microscale_control} to prove the renormalized commutator estimate Proposition~\ref{prop:renormalized_commutator_estimate}. This follows very similarly to~\cite[Proposition 4.1]{nguyen_mean-field_2022}. The bound~\eqref{eq:renormalized_commutator_estimate} follows by combining~\eqref{eq:old-commutator-estimate} with Lemma~\ref{lem:renormalized_Holder}.

\begin{proof}[Proof of Proposition~\ref{prop:renormalized_commutator_estimate}] First we prove~\eqref{eq:old-commutator-estimate}. We note that~\eqref{eq:periodic-correction} implies that
\begin{equation}\label{eq:kernek_lipschitz_bound}
|\nabla_x K_\psi(x,y)|\leq \|\nabla\psi\|_{L^\infty}|x-y|^{-\s-1}.\end{equation}
Adding and subtracting by $\delta_{x_i}^{(\eta)}$ we then find that
\[\int \K_\psi(x,y)\,\diff (\mu_N-\mu)^{\otimes 2}(x,y)=\text{Term}_1+\text{Term}_2+\text{Term}_3\]
where
\[\text{Term}_1:=\int \K_\psi(x,y)\,\diff (\mu_N^{(\eta)}-\mu)^{\otimes 2}(x,y)\]
\[\text{Term}_2:=-\frac{2}{N}\sum_{i=1}^N\int\K_\psi(x,y)\,\diff (\delta_{x_i}-\delta_{x_i}^{(\eta)})(x)\,\diff \mu(y)\]
\[\text{Term}_3:=\frac{1}{N^2}\sum_{1\leq i\neq j\leq N}^N\int_{(\T^d)\setminus\Delta}\K_\psi(x,y)\,\diff (\delta_{x_i}-\delta_{x_i}^{(\eta)})(x)\,\diff (\delta_{x_i}+\delta_{x_i}^{(\eta)})(y).\]

$\text{Term}_1$ is bounded using Propositions~\ref{prop:commutator-estimate} and~\ref{prop:modulated_monotonicity} by
\[A_\psi\|\mu_N^{(\eta)}-\mu\|_{\H}^2\leq A_\psi\Big(F_N(\ux_N^t)+C\|\mu\|_{L^\infty}\Big(\eta^2+\frac{\eta^{-\s}-\log\eta}{N}\Big)\Big).\]

Using~\eqref{eq:kernek_lipschitz_bound} and Young's convolution inequality we find that
\[|\text{Term}_2|\leq 2\eta \bigg\|\int \nabla_xK_\psi(\cdot ,y)\,\diff \mu(y)\bigg\|_{L^\infty}\leq C\|\nabla\psi\|_{L^\infty}\|\mu\|_{L^\infty}\eta.\]

$\text{Term}_3$ is the hardest to bound, however, it follows exactly as in~\cite[Proposition 4.1]{nguyen_mean-field_2022} with~\cite[Corollary 2.3]{nguyen_mean-field_2022} replaced with Corollary~\ref{cor:microscale_control} that
\[|\text{Term}_3|\leq C\|\nabla\psi\|_{L^\infty}\Big(F_N(\ux_N,\mu)+C\|\mu\|_{L^\infty}\Big(\frac{\eta^{-\s}+1}{N}+\frac{\eps^{-1}-\log\eps}{N}+\eps+\eta\eps^{-\s-1}\Big)\Big).\]
 We then conclude that~\eqref{eq:old-commutator-estimate} holds by optimizing over $\eta$ and $\eps$.

Moving on to showing~\eqref{eq:renormalized_commutator_estimate}, we first note that
\[\int_{(\T^d)^2\setminus \Delta} \K_\psi(x,y)\,\diff(\mu_N-\mu)\,\diff (\mu_N+\mu)=\int_{(\T^d)^2\setminus\Delta} \K_\psi(x,y)\,\diff(\mu_N-\mu)^{\otimes 2}+2\int_{(\T^d)^2} \K_\psi(x,y)\,\diff(\mu_N-\mu)\,\diff \mu\]

The inequality~\eqref{eq:old-commutator-estimate} in Proposition~\ref{prop:renormalized_commutator_estimate} bounds the first term by
\[CA_\psi\Big(F_N(\ux_N,\mu)+C\|\mu\|_{L^\infty}N^{-\beta}\Big).\]
This is in turn bounded by the right-hand side of~\eqref{eq:renormalized_commutator_estimate} since Lemma~\ref{lem:renormalized_Holder} and Young's inequality imply that
\[F_N(\ux_N,\mu)=H_N(\ux_N)-2\int\g(x-y) \,\diff \mu_N\,\diff \mu+\Ec(\mu)\leq C\Big( H_N(\ux_N)+\|\mu\|_{\H}+C\|\mu\|_{L^\infty}N^{-\beta} \Big).\]

To bound the second term, adding and subtracting by $\delta_{x_i}^{(\eta)}$ for $\eta=N^{-1/d}$, we have that
\[\int_{(\T^d)^2} \K_\psi(x,y)\,\diff(\mu_N-\mu)\,\diff \mu=\int_{(\T^d)^2} \K_\psi(x,y)\,\diff(\mu_N-\mu_N^{(\eta)})\,\diff \mu+\int_{(\T^d)^2} \K_\psi(x,y)\,\diff(\mu_N^{(\eta)}-\mu)\,\diff \mu.\]
The first term on the right-hand side is bounded by $C\|\nabla\psi\|_{L^\infty}\|\mu\|_{L^\infty}N^{-1/d}$ exactly like $\text{Term}_2$. Proposition~\ref{prop:commutator-estimate} and then Proposition~\ref{prop:modulated_monotonicity} bound the second term on the right-hand side by
\[C A_\psi\|\mu_N^{(\eta)}-\mu\|_{\H}(\|\mu\|_{\H}+1)\leq C A_\psi(F_N(\ux_N,\mu)+C\|\mu\|N^{-\beta})^{1/2}(\|\mu\|_{\H}+1).\]
Altogether these conclude the claimed bound.

\end{proof}

{\small
\bibliographystyle{alpha}
\bibliography{bibliography}
}

\end{document}